\documentclass[reqno, 11pt]{amsart}
\allowdisplaybreaks

\usepackage{amsmath,amssymb,amsthm,amssymb,amsfonts,amstext,mathtools}
\usepackage{url}
\usepackage[ruled,linesnumbered,boxed,boxruled,vlined]{algorithm2e}
\usepackage{algorithmic}
\usepackage{color}
\usepackage{verbatim}
\usepackage{hhline}
\usepackage{diagbox}
\usepackage{mleftright}
\usepackage{relsize}
\usepackage[normalem]{ulem}
\usepackage{amsopn}
\usepackage{textcomp}
\usepackage{tikz}
\usetikzlibrary{matrix}
\usepackage{slashbox}
\usepackage{etoolbox}
\usepackage{setspace}
\usepackage[colorlinks=true, urlcolor=blue, linkcolor=blue, citecolor=blue]{hyperref}
\usepackage[left=28mm, right=28mm, top=30mm, bottom=30mm, includefoot, includehead]{geometry}

\AtBeginEnvironment{algorithm}{\setstretch{1}}
\newtheorem{theorem}{Theorem}[section]
\newtheorem{corollary}[theorem]{Corollary}
\newtheorem{definition}[theorem]{Definition}

\newtheorem{lemma}[theorem]{Lemma}
\newtheorem{proposition}[theorem]{Proposition}
\newtheorem{remark}[theorem]{Remark}
\newtheorem{example}[theorem]{Example}
\newtheorem{conjecture}[theorem]{Conjecture}
\newtheorem{question}[theorem]{Question}
\numberwithin{equation}{section}

\newcommand{\C}{\mathbb{C}}
\newcommand{\R}{\mathbb{R}}

\newcommand{\N}{\mathbb{N}}

\newcommand{\sM}{\mathcal{M}}

\DeclarePairedDelimiter{\ceil}{\lceil}{\rceil}
\DeclareMathOperator{\rank}{rank}

\makeatletter
\@namedef{subjclassname@2020}{%
	\textup{2020} Mathematics Subject Classification}
\makeatother

\title{Exact solutions in low-rank approximation with zeros}

\author{Kaie Kubjas}
\author{Luca Sodomaco}
\address{Department of Mathematics and Systems Analysis, Aalto University}
\email{kaie.kubjas@aalto.fi}
\email{luca.sodomaco@aalto.fi}
\author{Elias Tsigaridas}
\address{Inria Paris and IMJ-PRG, Sorbonne Universit\'e and Paris Universit\'e}
\email{elias.tsigaridas@inria.fr}


\begin{document}

\maketitle

\begin{abstract}
Low-rank approximation with zeros aims to find a matrix of fixed rank and with a fixed zero pattern that minimizes the Euclidean distance to a given data matrix.
We study the critical points of this optimization problem using algebraic tools.
In particular, we describe special linear, affine, and determinantal relations satisfied by the critical points.
We also investigate the number of critical points and how this number is related to the complexity of nonnegative matrix factorization problem.
\end{abstract}

\section{Introduction}
The {\em best rank-$r$ approximation problem} aims to find a real rank-$r$ matrix that minimizes the Euclidean distance to a given real data matrix.
The solution of this problem is completely addressed by the Eckart-Young-Mirsky theorem which states that the best rank-$r$ approximation is given by the first $r$ components of the {\em singular value decomposition} (SVD) of the data matrix.

We study the {\em structured best rank-$r$ approximation problem}, namely we consider additional linear constraints on rank-$r$ matrices.
We focus on {\em coordinate subspaces}, i.e., linear spaces that are defined by setting some entries to zero.
Let $S\subset[m]\times[n]$ denote the indices of zero entries.
Given $U=(u_{ij})\in\R^{m\times n}$, our optimization problem becomes
\begin{equation} \label{low_rank_approximation_with_zeros}
\begin{aligned}
\min_{X} \  & d_U(X)\coloneqq\sum_{i=1}^m \sum_{j=1}^n (x_{ij}-u_{ij})^2\\
\textrm{s.t.} \  & \text{$x_{ij}=0\ \forall (i,j)\in S$ and $\rank(X)\leq r$}\,.
\end{aligned}
\end{equation}

Structured low-rank approximation problem has been studied in \cite{chu2003structured,markovsky2008structured,markovsky2012low};
see also \cite{GilGli-low-2011} for low rank approximations with weights.
Exact solutions to this problem have been investigated by Golub, Hoffman and Stewart~\cite{golub1987generalization}, and by Ottaviani, Spaenlehauer and Sturmfels \cite{OSS-esslra-14}.\\
In~\cite{golub1987generalization}, rank-$r$ critical points are studied under the constraint that entries in a set of rows or in a set of columns of a matrix stay fixed.
This situation is more general than ours in the aspect that the fixed entries are not required to be zero but more restrictive when it comes to the indices of the entries that are fixed.
In~\cite{OSS-esslra-14}, rank-$r$ critical points restricted to generic subspaces of matrices are studied.
In our paper, the linear spaces set some entries equal to zero and hence are not generic.
Because of this, we cannot use many powerful tools from algebraic geometry and intersection theory and we have to come up with algebraic and computational techniques that exploit this special structure.
For some properties of determinantal ideals of matrices with 0 entries and their relations to problems in graph theory we refer the reader to \cite{ConWel-lovasz-19} and references therein.
Horobet and Rodriguez study the problem  when at least one solution of a certain family of optimization problems satisfies given polynomial conditions, and address the structured low-rank approximation as a particular case~\cite[Example 15]{horobet2020data}.

The global minimum of the optimization problem~(\ref{low_rank_approximation_with_zeros}) always exists, because we can select any point $X$ in the feasible region and consider the feasible region intersected with the closed ball $B_{d_U(X)}(U)$ centered at $U$ and with radius $d_U(X)$. Since the feasible region is a closed semialgebraic set, then the intersection is closed and bounded, and hence compact. The distance function is continuous, thus achieves its minimum on this set. This minimum is a global minimum of~(\ref{low_rank_approximation_with_zeros}). The optimization problem~(\ref{low_rank_approximation_with_zeros}) is nonconvex and often local methods are used to solve it.
They return a local minimum of the optimization problem.  
There are heuristics for finding a global minimum, but these heuristics do not guarantee that a local minimum is indeed a global minimum.
We refer to \cite{markovsky2008structured} for various algorithms and to \cite{schost2016quadratically}
for an algorithm with locally quadratic convergence.
Cifuentes recently introduced convex relaxations for structured low-rank approximation that under certain assumptions have provable guarantees~\cite{cifuentes2021convex}.
Another interesting direction, closely related but not directly applicable to our problem, is to employ recent optimization techniques for simultaneously sparse and low rank approximation \cite{ParSel-improved-17,SalRicVay-icml-012}.

To compute a global minimizer of (\ref{low_rank_approximation_with_zeros}) algebraically, we need to look at all the complex critical points of the polynomial function $d_U\colon\C^{m\times n}\to\C$ on the intersection $\mathcal{L}_r^S\coloneqq\mathcal{X}_r\cap\mathcal{L}^S$, where
\begin{align}\label{eq: def L r S}
\mathcal{X}_r \coloneqq\{X\in\C^{m\times n}\mid\rank(X)\leq r\}\,,\ \mathcal{L}^S \coloneqq\{X\in\C^{m\times n}\mid x_{ij}=0\ \forall (i,j)\in S\},
\end{align}
and then select the real solution that minimizes the Euclidean distance. The problem of finding critical points of $d_U$ on $\mathcal{L}_r^S$ can be considered in the more general setting when $U$ is a complex data matrix. This setting includes the practically meaningful setting when $U$ has real entries.
If $U\in\C^{m\times n}$ is generic, namely if it belongs to the complement of a Zariski closed set, then the number of critical points is constant and is called the {\em Euclidean Distance degree} (ED degree) of $\mathcal{L}_r^S$.
We denote this invariant by $\mathrm{EDdegree}(\mathcal{L}_r^S)$.
The importance of the ED degree is that it measures the algebraic complexity of writing the optimal solution as a function of $U$.
More generally, the ED degree of an algebraic variety is introduced in \cite{draisma2016euclidean}.
The main goal of this paper is to study the critical points and the ED degree of the minimization problem (\ref{low_rank_approximation_with_zeros}).

\medskip
When rank is one, then characterizing critical points becomes a combinatorial problem.
More precisely, listing all critical points translates to the problem of listing minimal vertex covers of a bipartite graph. The complexity of counting vertex covers in a bipartite graph is known to be {\texttt\#}P-complete~\cite{provan1983complexity}. Our main result about rank-one critical points is Proposition~\ref{prop: ED degree rank 1 zero pattern} which gives the ED degree of $\mathcal{L}_1^S$ in terms of the minimal covers. For row/column and diagonal zero patterns this results in explicit formulas (Corollaries~\ref{corol: rank-one row and column zero patterns} and~\ref{corol: ED degree rank 1 diagonal}).

Our first main result for rank-$r$ critical points is Theorem~\ref{thm: linear_space_for_rank_r_critical_points} which studies the linear span of rank-$r$ critical points of $d_U$.
We call it the {\em critical space} in the structured setting. This is motivated by the notion of critical space of a tensor in the unstructured setting defined by Draisma, Ottaviani and Tocino \cite{DOT}.
From the algebraic perspective, Theorem~\ref{thm: linear_space_for_rank_r_critical_points} provides a lower bound on the minimal number of generators of degree one in the zero dimensional ideal of rank-$r$ critical points of $d_U$. When $\mathcal{L}_r^S$ is an irreducible variety, we expect this lower bound to be also an upper bound, as stated in Conjecture \ref{conj: linear span}.

In the unstructured setting, the rank-one critical points form a basis of the critical space and the rank-$r$ critical points are linear combinations of the basis vectors with coefficients in $\{0,1\}$.
In the structured setting, there are not enough rank-one critical points to give a basis of the critical space.
We leave it as an open question, whether there is a natural extension to a basis and whether the coefficients that give rank-$r$ critical points as linear combinations of basis elements have a nice description.

Our second main result is Proposition~\ref{prop:affine_relations} that describes affine linear relations that are satisfied by the rank-$r$ critical points of $d_U$ in the unstructured setting.
In the structured setting, we conjecture the affine linear relations satisfied by the rank-$r$ critical points of $d_U$.
The last kind of constraints satisfied by the rank-$r$ critical points that we consider are nonlinear determinantal constraints given in Proposition~\ref{prop: nonlinear relations minors constrained}.
The ED degree of $d_U$ is studied in Section~\ref{sec: EDdegree}.
Our experiments indicate that the ED degree is exponential in $|S|$.

\medskip
The optimization problem \eqref{low_rank_approximation_with_zeros} is motivated by the {\em nonnegative matrix factorization} (NMF) problem.
Given a nonnegative matrix $X \in \mathbb{R}^{m\times n}_\mathsmaller{\ge 0}$, the {\em nonnegative rank} of $X$ is the smallest $r$ such that
\[
X=AB, \text{ where } A \in \R^{m \times r}_\mathsmaller{\ge 0} \text{ and } B \in \R^{r \times n}_\mathsmaller{\ge 0}\,.
\]
NMF aims to find a matrix $X$ of nonnegative rank at most $r$ that minimizes the Euclidean distance to a given data matrix $U \in \mathbb{R}^{m\times n}_\mathsmaller{\ge 0}$, see~\cite{gillis2020nonnegative} for further details.

In Section~\ref{sec: nonnegative}, we apply the structured best rank-two approximation problem to NMF.
Let  $\sM_2$ be the set of matrices of nonnegative rank at most two and consider a matrix $U\in \R^{3 \times 3}_\mathsmaller{\ge 0}$.
In order to compute the best nonnegative rank-$2$ approximation of $U$, we need to compute the critical points of the Euclidean distance function $d_U$ over $\sM_2\cap\mathcal{L}^S$ for all zero patterns $S\subset[3]\times[3]$. 
We show that the minimal number of critical points needed to determine the global minimum of $d_U$ is 756 for a generic $U$.
For the same case, we show experimentally that the optimal critical point may have a few zeros.

\medskip
The rest of the paper is organized as follows.
In Section \ref{sec: prelim} we set our notations (Section~\ref{sec: notation}), we recall the basics of ED minimization on an algebraic variety (Section~\ref{sec: ED optimization}) and we discuss Frobenius distance minimization on a variety of low-rank matrices (Section~\ref{sec: unstructured low-rank approximation}).
In Section \ref{sec: rank one constrained approx} we address the best rank-one approximation problem with assigned zero patterns (Section~\ref{subsec:rank-one}) and best rank-$r$ approximation for rectangular and block diagonal matrices (Section~\ref{sec:rectangular-matrices}).
In Section \ref{sec: special relations} we investigate special polynomial relations among the critical points of $d_U$.
In particular, in Sections~\ref{sec:linear_relations} and~\ref{sec: affine relations} we concentrate on particular linear and affine relations among critical points respectively, and in Section~\ref{sec: special nonlinear relations} on some special nonlinear relations.
Observations for generic linear constraints not necessarily coming from assigned zero patterns are given in Section~\ref{sec: general subspaces}.
In Section \ref{sec: EDdegree} we provide conjectural ED degree formulas for special formats and zero patterns $S$, obtained from computational experiments.
In Section \ref{sec: nonnegative} we relate the minimization problem \eqref{low_rank_approximation_with_zeros} to nonnegative matrix factorization.
The results of Sections \ref{sec: EDdegree} and \ref{sec: nonnegative} are supported by computations that use the HomotopyContinuation.jl \cite{HomotopyContinuation.jl} software package as well as the software Macaulay2 \cite{M2} and Maple\texttrademark~2016 \cite{maple}.
The code can be found at \href{https://github.com/kaiekubjas/exact-solutions-in-low-rank-approximation-with-zeros}{github.com/kaiekubjas/exact-solutions-in-low-rank-approximation-with-zeros}.

\section{Preliminaries}\label{sec: prelim}

The preliminaries section consists of three subsections on algebra basics and notations (Section~\ref{sec: notation}), Euclidean distance minimization (Section~\ref{sec: ED optimization}) and unstructured low-rank approximation (Section~\ref{sec: unstructured low-rank approximation}).

\subsection{Algebra basics and notation}\label{sec: notation}

We start by setting up the notations used throughout the paper.
We always work over the field $\mathbb{F}=\R$ or $\mathbb{F}=\C$.
Without loss of generality, we always assume that $m\le n$ when considering the vector space $\mathbb{F}^{m\times n}$. 
Given an $m\times n$ matrix $X\in\mathbb{F}^{m\times n}$, we consider the subsets $I\subset[m]$ and $J\subset[n]$. We always assume that the elements of $I$ and $J$ are ordered in increasing order. We denote by $X_{I,J}$ the submatrix obtained by selecting the rows of $X$ with indices in $I$ and the columns of $X$ with indices in $J$. Moreover, if $|I|=|J|$ we denote by $M_{I,J}(X)$ the minor of $X$ corresponding to rows in $I$ and columns in $J$. If $I=J=\emptyset$, we set $M_{\emptyset,\emptyset}(X)\coloneqq 1$.

We use the notation $\mathbb{F}[X]$ to denote the ring of polynomials with variables $x_{ij}$ of $X$ and with coefficients in $\mathbb{F}$. For a given subset $Y\subset\mathbb{F}[X]$, we indicate with $\mathcal{V}(Y)$ the {\em (algebraic) variety} of $Y$, that is the zeros of $Y$ in $\mathbb{F}^{m\times n}$. In particular $\mathcal{V}(Y)=\mathcal{V}(\langle Y\rangle)$, where $\langle Y\rangle\subset\mathbb{F}[X]$ is the ideal generated by elements of $Y$. The subsets of $\mathbb{F}^{m\times n}$ of the form $\mathcal{V}(Y)$ for some $Y\subset\mathbb{F}[X]$ constitute the closed sets of the Zariski topology of $\mathbb{F}^{m\times n}$. In particular, a closed set $\mathcal{V}(Y)$ is {\em irreducible} if it is not the union of two non-empty proper subsets that are closed in the Zariski topology. Finally, for a given subset $Z\subset\mathbb{F}^{m\times n}$, we denote by $\mathcal{I}(Z)$ the {\em ideal of $Z$}, namely the ideal generated by all polynomials vanishing at the points of $Z$. We refer the reader to \cite{cox1992ideals} for a more general commutative algebra background.

\subsection{Euclidean distance minimization on a real algebraic variety}\label{sec: ED optimization}

Let $V^\mathsmaller{\R}$ be a finite dimensional real vector space equipped with a positive definite symmetric bilinear form $\langle\cdot,\cdot\rangle\colon V^\mathsmaller{\R}\times V^\mathsmaller{\R}\to\R$.
Without loss of generality, we can assume that $\langle\cdot,\cdot\rangle$ is the standard Euclidean inner product, that is $\langle x,y\rangle=\sum_ix_iy_i$, for $x, y \in V^\mathsmaller{\R}$.
Consider a real algebraic variety $\mathcal{X}^\mathsmaller{\R}\subset V^\mathsmaller{\R}$.
The generalized version of \eqref{low_rank_approximation_with_zeros} is the problem of minimizing the squared Euclidean distance $d_u(\cdot)\coloneqq\langle u-\cdot,u-\cdot\rangle$ over $\mathcal{X}^\mathsmaller{\R}$.
For a generic data point $u$, the minimum of $d_u$ is attained at a smooth point $x\in\mathcal{X}^\mathsmaller{\R}$, such that the vector $u-x$ is orthogonal to the tangent space $T_x\mathcal{X}^\mathsmaller{\R}$ with respect to $\langle\cdot,\cdot\rangle$.
Such a point $x$ is called a {\em critical} point of the function $d_u$ on $\mathcal{X}^\mathsmaller{\R}$.
We can formulate the whole problem algebraically, since all relations among critical points are described by polynomials in the coordinates of $x$ and $u$.

A complete study of the ideal of relations among critical points of $d_u$ requires us to work over an algebraically closed field.
For this reason we extend $V^\mathsmaller{\R}$ to a complex vector space $V=V^\mathsmaller{\R}\otimes\C$ and we let $\mathcal{X}$ be the complex variety associated to $\mathcal{X}^\mathsmaller{\R}$.
Finally, we denote again with $d_u$ the complex-valued function defined by $d_u(x)\coloneqq\langle u-x,u-x\rangle$.
We stress that $\langle\cdot,\cdot\rangle$ is not the standard inner product in the complex vector space $V$.
After this extension, one might study all complex critical points of the function $d_u$ on $\mathcal{X}$ and then look for the real global minimizer among all the complex solutions.
The interesting property is that the number of complex critical points of $d_u$ is an invariant of $\mathcal{X}$ for a generic data point $u\in V$ and is called the {\em Euclidean Distance degree} (ED degree) of $\mathcal{X}$. It is studied in detail in \cite{draisma2016euclidean}.

The bilinear form $\langle\cdot,\cdot\rangle$ induces a smooth quadric hypersurface $Q\subset V$ of equation $\langle x,x\rangle=0$.
It is called the {\em isotropic quadric} of $V$.
In particular, for every $x\in Q$, the tangent space $T_xQ$ is the hyperplane of vectors $y\in V$ such that $\langle x,y\rangle=0$.
An affine subspace $W\subset V$ is {\em transversal} to $Q$ if for every $x\in W\cap Q$, the span between $T_xQ$ and $W$ generates $V$.
In such case, the space $V$ is the direct sum between $W$ and its orthogonal complement.
In turn, this yields a good notion of {\em orthogonal projection} $\pi_W\colon V\to W$ onto $W$ with respect to $Q$.
We refer to \cite[Section 4]{piene2015polar} and \cite[Section 1.1]{sodphd} for more details on orthogonal projections with respect to the isotropic quadric $Q$.

The following two basic lemmas are needed in the following sections.

\begin{lemma}\label{lemma: critical points projection}
In the setting explained before, consider a proper affine subspace $W\subset V$ transversal to the isotropic quadric $Q$ and an algebraic variety $\mathcal{X}\subset W$.
Let $\pi_W\colon V\to W$ be the orthogonal projection onto $W$. Then, for all $u\in V$, the critical points on $\mathcal{X}$ of the squared distance functions $d_u$ and $d_{\pi_W(u)}$ coincide.
\end{lemma}

\begin{proof}
Let $x\in\mathcal{X}$ be a critical point of $d_u$. In particular $\langle u-x,y\rangle=0$ for all $y\in T_x\mathcal{X}$.
Furthermore $\langle\pi_W(u)-x,y\rangle=\langle u-x,y\rangle-\langle u-\pi_W(u),y\rangle=0-0=0$ since $u-\pi_W(u)$ and $y$ sit in orthogonal subspaces.
\end{proof}

\begin{lemma}\label{lemma: ED degree union}
Consider the affine varieties $\mathcal{X}_1,\ldots,\mathcal{X}_p$ in $V$ and suppose that $\mathcal{X}_i\not\subset \mathcal{X}_j$ for all $i\neq j$.
Then $\mathrm{EDdegree}(\mathcal{X}_1\cup\cdots\cup\mathcal{X}_p)=\sum_{i=1}^p\mathrm{EDdegree}(\mathcal{X}_i)$.
\end{lemma}

\begin{proof}
Let $u\in V$ be a generic data point and $\mathcal{X}=\mathcal{X}_1\cup\cdots\cup\mathcal{X}_p$.
A smooth point $x\in\mathcal{X}$ is critical for $d_u$ if the vector $u-x$ is orthogonal to the tangent space $T_x\mathcal{X}$.
Since $x$ is a smooth point of $\mathcal{X}$, there exists a unique index $i\in[p]$ such that $x$ is a smooth point of $\mathcal{X}_i$ and $T_x\mathcal{X}=T_x\mathcal{X}_i$.
Here we are using the hypothesis $\mathcal{X}_i\not\subset \mathcal{X}_j$ for all $i\neq j$.
Therefore $x$ is a critical point for $d_u$ on $\mathcal{X}_i$.
This yields the inequality $\mathrm{EDdegree}(\mathcal{X})\le\sum_{i=1}^p\mathrm{EDdegree}(\mathcal{X}_i)$.

To prove the reverse inequality, consider an index $i\in[p]$ and a critical point $x\in\mathcal{X}_i$ for $d_u$.
If $u$ is sufficiently generic, then $x$ lies outside the singular locus of $\mathcal{X}$ and $T_x\mathcal{X}_i=T_x\mathcal{X}$.
Therefore $x$ is a critical point for $d_u$ on $\mathcal{X}$.
Repeating the argument for every $i\in[p]$, we derive the desired inequality $\mathrm{EDdegree}(\mathcal{X})\ge\sum_{i=1}^p\mathrm{EDdegree}(\mathcal{X}_i)$.
\end{proof}

\subsection{Unstructred low-rank approximation}\label{sec: unstructured low-rank approximation}

We consider the notation of Section \ref{sec: ED optimization}. We specialize to the Euclidean space $(V^\mathsmaller{\R},\langle\cdot,\cdot\rangle)=(\R^{m\times n},\langle\cdot,\cdot\rangle_F)$, where $\langle\cdot,\cdot\rangle_F$ is the {\em Frobenius inner product} in $\R^{m\times n}$.
It is defined by $\langle A, B\rangle_F\coloneqq\mathrm{trace}(AB^T)$ for any two $m\times n$ matrices $A$ and $B$.
Given $U\in\R^{m\times n}$, we minimize the squared Frobenius distance
\[
d_U(X)=\langle X-U,X-U\rangle_F=\sum_{i=1}^m\sum_{j=1}^n(x_{ij}-u_{ij})^2\,,
\]
over the real variety $\mathcal{X}^\mathsmaller{\R}=\mathcal{X}_r^\mathsmaller{\R}$ of $m\times n$ matrices of rank at most $r$.
Note that $d_U(X)$ coincides with the squared Euclidean distance between $X$ and $U$ in $\R^{m\times n}$.

An important tool to study low-rank approximation problem is given by the following decomposition of a real matrix.

\begin{theorem}[Singular Value Decomposition]\label{thm: SVD}
Any matrix $U\in\R^{m\times n}$ admits the {\em Singular Value Decomposition (SVD)}
\begin{equation}\label{eq: free SVD}
U=A\,\Sigma\,B^T\,,
\end{equation}
where $A\in\R^{m\times m}$ and $B\times\R^{n\times n}$ are orthogonal matrices and $\Sigma\in\R^{m\times n}$ is such that $\Sigma_{ii}=\sigma_i$ for some real numbers $\sigma_1\ge\sigma_2\ge\cdots\sigma_m\ge 0$, otherwise $\Sigma_{ij}=0$.
The numbers $\sigma_i$ are called {\em singular values} of $U$.
Denoting by $a_i$ and $b_i$ the columns of $A$ and $B$ respectively, for all $i\in[m]$ the pair $(a_i,b_i)$ is called a {\em singular vector pair} of $U$.
If the singular values are all distinct, then all singular vector pairs are unique up to a simultaneous change of sign.
\end{theorem}

\begin{remark}\label{rmk: Algebraic SVD}
Theorem \ref{thm: SVD} extends to complex matrices using unitary matrices and their conjugates, but complex conjugation is not an algebraic operation. If a matrix $U\in\C^{m\times n}$ factors as in \eqref{eq: free SVD}, then $U$ admits an {\em algebraic SVD}. The complex matrices admitting an algebraic SVD are characterized in \cite[Theorem 2 and Corollary 3]{CH}, see also \cite[Section 3]{DLOT}.
\end{remark}

Similarly as in Section \ref{sec: ED optimization}, we rephrase our low-rank approximation problem over $\C$. In particular, we work in the space $\C^{m\times n}$ and we compute all critical points of the complex-valued function $d_U(X)=\langle U-X,U-X\rangle_F$ on the variety $\mathcal{X}=\mathcal{X}_r$ of complex $m\times n$ matrices of rank at most $r$.
The following result characterizes such critical points when the data matrix $U$ is real.
A similar statement holds for every complex data matrix $U$ admitting an algebraic SVD.
We report the statement used in \cite[Theorem 2.9]{MR3456581}.

\begin{theorem}[Eckart-Young-Mirsky]\label{thm: Eckart-Young}
Consider a matrix $U\in\R^{m\times n}$ of rank $k$ and its SVD as in (\ref{eq: free SVD}).
Let $r\in[k]$.
Then all the critical points of $d_U$ on $\mathcal{X}_r$ are of the form
\begin{equation}\label{eq: critical points EYM}
A(\Sigma_{i_1}+\cdots+\Sigma_{i_r})B^T
\end{equation}
for all subsets $\{i_1<\cdots<i_r\}\subset[k]$, where $\Sigma_j$ is the $m\times n$ matrix whose only non-zero entry is $\Sigma_{j,j}=\sigma_j$.
If the non-zero singular values of $U$ are distinct, then there are $\binom{k}{r}$ critical points.
\end{theorem}

Therefore, Theorem \ref{thm: Eckart-Young} solves the best rank-$r$ approximation problem and the nice structure of critical points leads to various interesting consequences.
In particular, assuming that $U$ is full rank, their number is independent from the largest dimension $n$.
Moreover, their linear span does not depend on the rank $r$; it is studied in \cite{MR3456581,DOT} in the more general context of tensor spaces, see also Proposition \ref{prop: linear_space_for_rank_one_critical_points}.

Consider a data matrix $U\in\R^{m\times n}$ and let $Z_{U,r}$ be the set of critical points of $d_U$ on $\mathcal{X}_r$.
We denote by $\langle Z_{U,r}\rangle$ the linear span of $Z_{U,r}$ in $\R^{m\times n}$.
A consequence of Theorem \ref{thm: Eckart-Young} is that $\langle Z_{U,r}\rangle$ does not depend on the rank $r$, namely $\langle Z_{U,r}\rangle=\langle Z_{U,1}\rangle$ for all $r\in[m]$.
The following result is a special case of \cite[Theorem 1.1]{DOT} and gives the equations of $\langle Z_{U,1}\rangle$.

\begin{proposition}\label{prop: linear_space_for_rank_one_critical_points}
Given $U\in\R^{m\times n}$, the ideal $\mathcal{I}(Z_{U,1})=\mathcal{I}(\langle Z_{U,1}\rangle)$ is generated by the linear forms in $\R[X]_1$
\begin{equation}\label{eq: linear relations unconstrained case}
r_U^{(i,j)}(X)\coloneqq(XU^T-UX^T)_{ij}\ \forall\,i,j\in[m]\,,\quad c_U^{(i,j)}(X)\coloneqq(X^TU-U^TX)_{ij}\ \forall\,i,j\in[n]\,.
\end{equation}
In particular, if $U$ is sufficiently generic, then (assuming $m\le n$)
\begin{enumerate}
    \item the $\binom{m}{2}$ forms $r_U^{(i,j)}(X)$ are linearly independent,
    \item $m(n-1)-\binom{m}{2}$ of the forms $c_U^{(i,j)}(X)$ are linearly independent,
    \item none of the forms $r_U^{(i,j)}(X)$ is linear combination of some of the forms $c_U^{(i,j)}(X)$, and viceversa.
\end{enumerate}
Therefore $\dim(\langle Z_{U,1}\rangle)=m$ if $U$ is sufficiently generic.
\end{proposition}

We explain an implication of Proposition~\ref{prop: linear_space_for_rank_one_critical_points} which is used in the proof of Corollary~\ref{corol: codim linear span}.

\begin{remark}\label{rmk: on linear relations with rows or columns}
Let $U\in\R^{m\times n}$ and let $U_{I,J}$ be a submatrix of $U$.
Consider the system
\[
r_{U_{I,J}}^{(i,j)}=0\quad\forall\,i,j\in[|I|]\,,\ c_{U_{I,J}}^{(i,j)}=0\quad\forall\,i,j\in[|J|]\,.
\]
If $|I|\le|J|$ and $U$ is sufficiently generic, then
\begin{enumerate}
    \item the first $\binom{|I|}{2}$ equations are linearly independent,
    \item there are $|I|(|J|-1)-\binom{|I|}{2}$ linearly independent equations in the second set,
    \item no equation in one set is linear combination of equations in the other set.
\end{enumerate}
\end{remark}

\section{Rank-one structured approximation and beyond}\label{sec: rank one constrained approx}

This section is divided into two subsections: In Section~\ref{subsec:rank-one}, we focus on rank-one approximation with zeros, and in Section~\ref{sec:rectangular-matrices}, on the simplest cases of rank-$r$  approximation with zeros for rectangular and block-diagonal matrices.

\subsection{Rank-1 structured approximation}\label{subsec:rank-one}

Rank-one approximation with zeros is built on the observation that non-zero entries of a rank-one matrix form a rectangular submatrix.
After introducing the notions of rectangular zero patterns in Definition~\ref{def: rectangular zero pattern} and minimal covers in Definition~\ref{def: minimal cover}, Proposition~\ref{prop: procedure rank-one approximate} gives the procedure for determining the best rank-one approximation in the structured setting. The main result in this subsection is Proposition~\ref{prop: ED degree rank 1 zero pattern} which gives the ED degree of $\mathcal{L}_1^S$ in terms of the minimal covers.
We recall that the variety $\mathcal{L}_r^S=\mathcal{X}_r\cap\mathcal{L}^S$ was defined in \eqref{eq: def L r S} for any $r\in[m]$.
For row/column and diagonal zero patterns, this results in explicit formulas (Corollaries~\ref{corol: rank-one row and column zero patterns} and~\ref{corol: ED degree rank 1 diagonal}). We end the section with Algorithm~\ref{algorithm: minimal_covers} for finding all minimal covers of a zero pattern $S$.

\begin{definition} \label{def: rectangular zero pattern}
We say that a zero pattern $S\subset [m]\times[n]$ is {\em rectangular} if the indices that are not in $S$ form a rectangular matrix.
More precisely, a rectangular zero pattern has the form $S=(S_1\times[n])\cup([m]\times S_2)$, for some $S_1\subset[m]$ and $S_2\subset[n]$.
Sometimes we denote this zero pattern also by $(S_1,S_2)$.
\end{definition}

\begin{definition} \label{def: minimal cover}
Let $S,T \subset [m]\times[n]$ be two zero patterns.
\begin{enumerate}
	\item The zero pattern $T$ is a {\em cover} of the zero pattern $S$, if $S\subset T$ and if $T$ is rectangular.
	\item The zero pattern $T$ is a {\em minimal cover} of $S$, if it is minimal among all covers of the zero pattern $S$ with respect to inclusion. We denote by $MC(S,m,n)$ the set of all minimal covers of $S\subset[m]\times[n]$.
\end{enumerate}
\end{definition}

\noindent For example, if $S=\{(1,1),(1,2),(2,2)\}$, then $MC(S,3,3)=\{([2],\emptyset),(\emptyset,[2]),(\{1\},\{2\})\}$ (see Figure \ref{fig: cover zero pattern}).

\begin{figure}[htbp]
    \centering
    \begin{tikzpicture}
    \matrix(m)[matrix of math nodes, left delimiter = (,right delimiter = ),row sep=5pt,column sep = 5pt, minimum width=width("-1")]{
    0&0&2\\
    1&0&1\\
    2&1&3\\};
    \draw (m-1-1.north west) -- (m-2-1.south west) -- (m-2-3.south east) -- (m-1-3.north east) -- (m-1-1.north west) -- cycle;
    \end{tikzpicture}
    \quad
    \begin{tikzpicture}
    \matrix(m)[matrix of math nodes,left delimiter = (,right delimiter = ),row sep=5pt,column sep = 5pt, minimum width=width("-1")]{
    0&0&2\\
    1&0&1\\
    2&1&3\\};
    \draw (m-1-1.north west) -- (m-3-1.south west) -- (m-3-2.south east) -- (m-1-2.north east) -- (m-1-1.north west) -- cycle;
    \end{tikzpicture}
    \quad
    \begin{tikzpicture}
    \matrix(m)[matrix of math nodes,left delimiter = (,right delimiter = ),row sep=5pt,column sep = 5pt, minimum width=width("-1")]{
    0&0&2\\
    1&0&1\\
    2&1&3\\};
    \draw (m-1-1.north west) -- (m-1-1.south west) -- (m-1-2.south west) -- (m-3-2.south west) -- (m-3-2.south east) -- (m-1-2.south east) -- (m-1-3.south east) -- (m-1-3.north east) -- (m-1-1.north west) -- cycle;
    \end{tikzpicture}
    \vspace*{-2mm}
    \caption{Minimal covers of a $3\times 3$ matrix with zero pattern $S=\{(1,1),(1,2),(2,2)\}$.}
    \label{fig: cover zero pattern}
\end{figure}

The main difficulty in studying the best rank-one approximation problem with assigned zero pattern $S$ lies in identifying all the minimal covers of $S$.

\begin{proposition}\label{prop: ED degree rank 1 zero pattern}
Let $S\subset[m]\times[n]$ be a zero pattern.
Consider the variety $\mathcal{X}_1$ of $m\times n$ matrices of rank at most one and the intersection $\mathcal{L}_{1}^S=\mathcal{X}_1\cap\mathcal{L}^S$. Then
\begin{equation}\label{eq: ED degree rank 1 zero pattern}
\mathrm{EDdegree}(\mathcal{L}_1^S)=\sum_{(A_r,A_c)\in MC(S,m,n)}\min(m-|A_r|,n-|A_c|)\,.
\end{equation}
\end{proposition}

\begin{proof}
There is a bijection between the irreducible components of $\mathcal{L}_{1}^S$ and the elements of $MC(S,m,n)$.
More precisely, for every pair $(A_r,A_c)\in MC(S,m,n)$, the corresponding component of $\mathcal{L}_{1}^S$ is isomorphic to the variety of $(m-|A_r|)\times(n-|A_c|)$ matrices of rank at most one, and by Theorem \ref{thm: Eckart-Young} this component has ED degree equal to $\min(m-|A_r|,n-|A_c|)$.
Moreover, given two distinct minimal covers $T_1$ and $T_2$ in $MC(S,m,n)$ and their corresponding irreducible components $\mathcal{V}_1$ and $\mathcal{V}_2$, then neither $\mathcal{V}_1\subset\mathcal{V}_2$ nor $\mathcal{V}_2\subset\mathcal{V}_1$.
The statement follows by Lemma \ref{lemma: ED degree union}.
\end{proof}

\begin{corollary} \label{prop: procedure rank-one approximate}
Let $S \subset [m] \times [n]$. Given $U\in \R^{m\times n}$, its best rank-one approximation with zeros in $S$ is found by the following procedure:
\begin{enumerate}
    \item Identify all minimal covers of $S$.
    \item Find the best rank-one approximation for each of the minimal covers of $S$.
    \item Choose the best rank-one approximation over all the minimal covers of $S$.
\end{enumerate}
\end{corollary}

\begin{corollary}
Let $S=\{(1,1)\}$. It has two minimal covers, i.e.,  $MC(S,m,n)=\{ (\{1\},\emptyset),(\emptyset,\{1\}) \}$. Given $U\in \R^{m\times n}$, its best rank-one approximation with a zero in $S$ is found by first identifying the best rank-one approximations for the two minimal covers of $S$ and then choosing out of the two rank-one matrices the one that minimizes the Euclidean distance to $U$.
\end{corollary}

\begin{definition}\label{def:mask-matrix}
Let $S\subset[m]\times[n]$ be a zero pattern.
The {\em mask matrix} of $S$ is the $m\times n$ matrix $M^S$ with entries in $\{0,1\}$ whose $(i,j)$-th entry is the characteristic function $\chi_S(i,j)$ of $S$.
\end{definition}

\begin{definition}
Two zero patterns $S_1$ and $S_2$ are said to be {\em permutationally equivalent} if there exist permutation matrices $P_1$ and $P_2$ such that we can write $M^{S_1} = P_1 M^{S_2} P_2$.
\end{definition}

\begin{corollary}\label{corol: rank-one row and column zero patterns}
Let $S\subset[m]\times[n]$ be a zero pattern which is permutationally equivalent to the row zero pattern $\{(1,1),\ldots,(1,|S|)\}$.
Then
\begin{equation}\label{eq: ED degree rank 1 zero pattern one row}
\mathrm{EDdegree}(\mathcal{L}_1^S)=\min\{m,n-|S|\}+\min\{m-1,n\}\,.
\end{equation}
Similarly, let $S\subset[m]\times[n]$ be a zero pattern which is permutationally equivalent to the column zero pattern $\{(1,1),\ldots,(|S|,1)\}$.
Then
\begin{equation}\label{eq: ED degree rank 1 zero pattern one column}
\mathrm{EDdegree}(\mathcal{L}_1^S)=\min\{m,n-1\}+\min\{m-|S|,n\}\,.
\end{equation}
\end{corollary}

\begin{corollary}\label{corol: ED degree rank 1 diagonal}
Let $S\subset[m]\times[n]$ be a zero pattern which is permutationally equivalent to the diagonal zero pattern $\{(1,1),\ldots,(|S|,|S|)\}$.
Then
\begin{equation}\label{eq: ED degree rank 1 zero pattern diagonal}
\mathrm{EDdegree}(\mathcal{L}_1^S)=\sum_{j=0}^{|S|}\binom{|S|}{j}\min\{m-j,n-|S|+j\}\,.
\end{equation}
\end{corollary}

Enumerating minimal covers of a zero pattern translates to the problem of enumerating minimal vertex covers of a bipartite graph.
A bipartite graph $G$ can be associated to a zero pattern of an $m \times n$-matrix $X=(x_{ij})$ in the following way: the bipartite graph $G$ has $m$ and $n$ vertices in the two parts, corresponding to the rows and columns of the matrix.
The edges of $G$ correspond to the zero entries of the matrix, i.e. $(i,j)\in E(G)$ if and only if $x_{ij}=0$.
A (minimal) cover of a zero pattern is then equivalent to a (minimal) vertex cover of the corresponding bipartite graph.
Since counting vertex covers in a bipartite graph is \texttt{\#}P-complete~\cite{provan1983complexity}, it follows that counting covers of a zero pattern is \texttt{\#}P-complete.
Recall that $\#P$ is the class of counting problems, where the goal is to count the number of solutions of a problem.
We know that $NP \subseteq \#P$ but we are not aware how high in the hierarchy $\#P$ is. We refer the interested reader to \cite[Chapter~9]{ArBa-cc-2009} for further details.

We suggest Algorithm~\ref{algorithm: minimal_covers}, that is based on dynamic programming, to find all minimal covers of a zero pattern $S$.
To simplify notation, we present the algorithm for the bipartite graph $G$ corresponding to the zero pattern $S$.
In particular, we use the following notation.
Let $G=(U,V,E)$, where $U=\{u_1,\ldots,u_m\}$ and $V=\{v_1,\ldots,v_n\}$ are the two parts of vertices.
For $u \in U$, we denote by $\mathcal{N}(u)$ the set of neighbors of $u$.
Let $U' \subset U$ and $V' \subset V$.
We denote by $G[U',V']$ the induced subgraph of $G$, i.e., the graph whose vertex set is $U' \cup V'$ and whose edge set is the subset of $E$ that consists of edges whose both endpoints are in $U' \cup V'$.
We consider the graph $G[U',V']$ as a bipartite graph with $U'$ and $V'$ being the two parts of vertices.

\begin{center}
\begin{minipage}{.75\linewidth}
\begin{algorithm}[H]
\caption{Minimal covers of a bipartite graph $G=(U,V,E)$}\label{algorithm: minimal_covers}
\begin{algorithmic}[1]
\STATE{{\bf procedure} \textsc{MinimalCovers}($G=(U,V,E)$)}
\quad\IF{$G$ is null graph}
\RETURN $\{(\emptyset,\emptyset)\}$
\ELSE
\STATE{$MC= \emptyset$}
\STATE{$MC_1= \textsc{MinimalCovers}(G[U\setminus\{u_1\},V])$}
\FOR{$(S_1,S_2) \in MC_1$}
\STATE{append $(S_1 \cup \{u_1\},S_2)$ to $MC$}
\ENDFOR
\STATE{$MC_2= \textsc{MinimalCovers}(G[U\setminus\{u_1\},V\setminus\mathcal{N}(u_1)])$}
\FOR{$(S_1,S_2) \in MC_2$}
\STATE{append $(S_1,S_2 \cup \mathcal{N}(u_1))$ to $MC$}
\ENDFOR
\RETURN $MC$
\ENDIF
\STATE{{\bf end procedure}}
\end{algorithmic}
\end{algorithm}
\end{minipage}
\end{center}

The following example illustrates that it is not enough, even to consider minimal covers with the least number of elements.

\begin{example}
Consider the $3\times 4$ matrix
\[
U =
\begin{pmatrix}
1 & -1 & -2 & -2\\
1 & 0 & 1 & -2\\
2 & 0 & 0 & 2
\end{pmatrix}\,.
\]
We look for the closest rank-one matrix to $U$ with zero pattern $S=\{(1,1),(1,2)\}$.
We have $MC(S,3,4)=\{(\{1\},\emptyset),(\emptyset,[2])\}$.
In particular, the first minimal cover consists of four elements, while the second minimal cover consists of six elements.
One verifies that the closest critical point to $U$ is of the second type and is equal to
\[
X =
\begin{pmatrix}
0 & 0 & -0.627896 & -2.36438\\
0 & 0 & -0.430261 & -1.62017\\
0 & 0 & 0.496139 & 1.86824
\end{pmatrix}\,.
\]
\end{example}

\subsection{Rank-\texorpdfstring{$r$}{r} structured approximation for rectangular and block-diagonal matrices} \label{sec:rectangular-matrices}

We finish this section with two results for structured best rank-$r$ approximation problem in the simplest cases where non-zero entries form rectangular and block-diagonal submatrices.

\begin{lemma}[Rectangular low-rank approximation]\label{remark: rectangular-rank-1-approximation}
Let $S \subseteq [m] \times [n]$ be such that $([m] \times [n]) \backslash S = T' \times T''$ for some $T' \subseteq [m]$ and $T'' \subseteq [n]$. Given $U\in \R^{m\times n}$, let $U' \in \R^{m\times n}$ be the matrix that is obtained from $U$ by setting the entries in $S$ equal to zero. The best rank-$r$ approximation of $U$ with zeros in $S$ is equal to the best (unstructured) rank-$r$ approximation of $U'$.
If the non-zero singular values of $U'$ are distinct, then by Theorem \ref{thm: Eckart-Young} there are $\binom{\min(|T'|,|T''|)}{r}$ critical points and they are given by the Singular Value Decomposition \eqref{eq: free SVD} of $U'$.
\end{lemma}

\begin{lemma}[Block-diagonal low-rank approximation]\label{remark: block-rank-1-approximation}
Let $S \subseteq [m] \times [n]$ be such that $([m] \times [n]) \backslash S = (T'_1 \times T''_1) \cup (T'_2 \times T''_2) \cup \ldots \cup (T'_s \times T''_s)$ for some non-empty pairwise disjoint $T'_1, T'_2, \ldots, T'_s \subseteq [m]$ and non-empty pairwise disjoint $T''_1, T''_2, \ldots, T''_s \subseteq [n]$. 
Consider all vectors $(r_1,\ldots,r_s) \in \N_0^s$ such that $\sum r_i=r$.
For a fixed vector $(r_1,\ldots,r_s)$, let $Z(r_1,\ldots,r_s)$ be the set of matrices such that for $1 \leq i \leq s$, the projection to $T'_i \times T''_i$ is a critical point of the unstructured rank-$r_i$ approximation problem for $U|_{T'_i \times T''_i}$. The union of $Z(r_1,\ldots,r_s)$ over all vectors $(r_1,\ldots,r_s) \in \N_0^s$ satisfying $\sum r_i=r$ gives the set of all critical points for the the structured best rank-$r$ approximation problem (\ref{low_rank_approximation_with_zeros}) for $U$.
\end{lemma}

\begin{example}
Consider the $3 \times 3$ matrix
$$
U = 
\begin{pmatrix}
1 & 2 & 3 \\
4 & 5 & 6 \\
7 & 8 & 9
\end{pmatrix}.
$$
Let $S=\{(1,3),(2,3),(3,1),(3,2)\}$. We look for the closest rank-two matrix to $U$ with zero pattern $S$. The support of this closest rank-two matrix is block-diagonal, i.e., the non-zero entries are  $([3]\times [3])\backslash S = (\{1,2\} \times \{1,2\}) \cup (\{3\} \times \{3\})$. There are $s=2$ blocks and three vectors $(r_1,r_2) \in \N_0^2$ satisfying $r_1+r_2=2$, namely the vectors $(2,0),(1,1)$ and $(0,2)$. The critical point corresponding to the vector $(2,0)$ is
$$
C_1 = 
\begin{pmatrix}
1 & 2 & 0 \\
4 & 5 & 0 \\
0 & 0 & 0
\end{pmatrix}.
$$
The critical points corresponding to the vector $(1,1)$ are
$$
C_2 = \begin{pmatrix}
1.3332 & 1.7455 & 0\\ 
3.8857 & 5.0873 & 0\\
0 & 0 & 9
\end{pmatrix}
\quad \text{and} \quad
C_3 = \begin{pmatrix}
-0.3332 & 0.2545 & 0\\ 
0.1143 & -0.0873 & 0\\
0 & 0 & 9
\end{pmatrix}.
$$
There are no critical points corresponding to the vector $(0,2)$, since the size of the last block is one and hence it cannot have rank two. The critical point that minimizes the squared Frobenius distance to $U$ is $C_2.$
\end{example}

\section{Special relations among critical points}\label{sec: special relations}

In this section we provide (some of) the generators of the ideal of critical points on $\mathcal{L}_r^S$ of $d_U$.
In particular, in Sections~\ref{sec:linear_relations} and~\ref{sec: affine relations} we concentrate on particular linear and affine relations among critical points respectively, and in Section~\ref{sec: special nonlinear relations} on some special nonlinear relations.
Observations for generic linear constraints not necessarily coming from assigned zero patterns are given in Section~\ref{sec: general subspaces}.

We stress that in our statements we always consider a real $m\times n$ matrix $U$. However, as we explained in Section \ref{sec: unstructured low-rank approximation}, our optimization problem is defined in the ambient space of $m\times n$ complex matrices, since in the structured setting the critical points of $d_U$ are not necessarily real.

\subsection{Linear relations among critical points} \label{sec:linear_relations}

Let $Z_{U,r}^S$ be the set of critical points of $d_U$ on $\mathcal{L}_r^S$. When $S=\emptyset$ we simply write $Z_{U,r}$ as in Section~\ref{sec: unstructured low-rank approximation}.
The study of the linear span $\langle Z_{U,r}^{S}\rangle$ is more involved when $S\neq\emptyset$.
The main result in this section is Theorem~\ref{thm: linear_space_for_rank_r_critical_points} which states that certain linear equations from the unstructured setting in Proposition~\ref{prop: linear_space_for_rank_one_critical_points} are satisfied by the rank-$r$ critical points of $d_U$ in the structured setting. In the rest of the section, we investigate when do these linear equations define $\langle Z_{U,r}^{S}\rangle$ (Conjecture~\ref{conj: linear span}, Examples~\ref{example: 3x3 one zero rank two}-\ref{example: 3x3 two zeros row basis}) and give a lower bound on the codimension of $\langle Z_{U,r}^{S}\rangle$ (Corollary~\ref{corol: codim linear span}).

\begin{definition}
Let $S\subset[m]\times[n]$ be a zero pattern.
Recall that $M^S$ denotes the mask matrix of $S$ as introduced in Definition~\ref{def:mask-matrix}.
We introduce the following equivalence relations $\sim_R^S $ and $\sim_C^S $ in the sets $[m]$ and $[n]$, respectively:
\begin{itemize}
\item $i\sim_R^S j$ if and only if the $i$-th and the $j$-th rows of $M^S$ coincide,
\item $i\sim_C^S j$ if and only if the $i$-th and the $j$-th columns of $M^S$ coincide.
\end{itemize}
\end{definition}

\begin{definition}\label{def: critical space}
We define {\em critical space} of $U\in\R^{m\times n}$ to be the linear space $H_U^{S}$ defined by relations
\begin{equation}\label{eq: linear relations constrained case}
\left\langle r_{U}^{(i,j)}\mid i\sim_R^S j\right\rangle+\left\langle c_{U}^{(i,j)}\mid i\sim_C^S j\right\rangle+\left\langle x_{ij}\mid(i,j)\in S\right\rangle\,.
\end{equation}
\end{definition}

Observe that $H_U^{S}$ does not depend on the rank $r$. Moreover, the previous definition is inspired by \cite[Definition 2.8]{DOT} in the context of unstructured low-rank tensor approximation.
In particular, for $S=\emptyset$ we denote the critical space simply by $H_U$.

Proposition \ref{prop: linear_space_for_rank_one_critical_points} tells us that $\langle Z_{U,r}\rangle=H_U$ for all $r\in[m-1]$.
The next result shows that at least one of the two inclusions in the last equality is still true when $S\neq\emptyset$.

\begin{theorem}\label{thm: linear_space_for_rank_r_critical_points}
Let $S\subset[m]\times[n]$, $U\in\R^{m\times n}$ and $r\in [m-1]$. Then $\langle Z_{U,r}^{S}\rangle\subset H_U^{S}$.
\end{theorem}

\begin{proof}
Let $L_1,\ldots,L_s$ be the constraints that set $s$ entries of the matrix to be equal to zero.
We denote by $\mathrm{Jac}_{\mathcal{X}_r}(X)$ and $\mathrm{Jac}_{\mathcal{L}^S}(X)$ the Jacobian matrices of $\mathcal{X}_r$ and $\mathcal{L}^S$ evaluated at $X$, respectively.
The rank-$r$ critical points $X\in\mathcal{L}_r^S$ of $d_U$ satisfy the equality constraints
\begin{equation}\label{eq: system Lagrange}
\begin{cases}
M_{I,J}(X)=0 & \forall\,|I|=|J|=r+1\\
L_k(X)=0 & \forall\,k\in[s]
\end{cases}
\quad
\mleft[
\begin{array}{c|c|c}
\lambda & \mu  & 1
\end{array}
\mright]
\mleft[
\begin{array}{c}
\mathrm{Jac}_{\mathcal{X}_r}(X)\\
\hline
\mathrm{Jac}_{\mathcal{L}^S}(X)\\
\hline
X-U
\end{array}
\mright]
=
\mleft[
\begin{array}{c|c|c}
0 & 0 & 0
\end{array}
\mright]\,,
\end{equation}
where $\lambda=(\lambda_{I,J})_{I,J}$ and $\mu=(\mu_1,\ldots,\mu_s)$ are vectors of Lagrange multipliers.

We denote by $v$ the vector of polynomials that is obtained when multiplying the vector and the augmented Jacobian matrix in \eqref{eq: system Lagrange}.
Its entries are naturally indexed by $(1,1)$,\ldots,$(m,n)$.
Let
\[
X^{i \leftrightarrow j}\coloneqq
\begin{bmatrix}
0 & \cdots & 0 & -x_{j1} & \cdots & -x_{jn} & 0 & \cdots & 0 & x_{i1} & \cdots & x_{in} & 0 & \cdots & 0\\
\end{bmatrix}^T
\]
be the vector with the entry $-x_{jk}$ at the position $(i,k)$ and the entry $x_{ik}$ at the position $(j,k)$ for all $k\in[n]$.

We show that $v \cdot X^{i \leftrightarrow j}$ is equal to a linear constraint in \eqref{eq: linear relations constrained case} plus some $(r+1)$-minors $M_{I,J}(X)$ and linear constraints $L_k(X)$ multiplied with Lagrange multipliers.
To do this, we study the products of the rows of the augmented Jacobian with the vector $X^{i \leftrightarrow j}$.

First, observe that the last row of the augmented Jacobian multiplied with $X^{i \leftrightarrow j}$ is precisely the linear form $r_{U}^{(i,j)}$.
Secondly, we show that the rows of the augmented Jacobian corresponding to minors multiplied with $X^{i \leftrightarrow j}$ are either zero or a sum of $(r+1)$-minors.
Let $A=\{a_1,\ldots,a_{r+1}\}\subset[m]$ and $B=\{b_1,\ldots,b_{r+1}\}\subset[n]$.
We consider the product
\begin{equation} \label{eq: jacobian_of_a_minor}
\begin{bmatrix}
\frac{\partial M_{A,B}(X)}{\partial x_{i1}} & \cdots & \frac{\partial M_{A,B}(X)}{\partial x_{in}}
\end{bmatrix}
\cdot X^{(j)}\,.
\end{equation}
If $i \notin A$, then the product~(\ref{eq: jacobian_of_a_minor}) is equal to zero.
Otherwise $i \in A$ and the product~(\ref{eq: jacobian_of_a_minor}) can be seen as the Laplace expansion of the matrix with rows in $(A\setminus\{i\})\cup\{j\}$ considered as a multiset and columns in $B$.
Hence if $j\notin A$, then the product~(\ref{eq: jacobian_of_a_minor}) is equal to the minor corresponding to rows in $(A\setminus\{i\})\cup\{j\}$ and columns in $B$.
Finally, if $j\in A$, then the product~(\ref{eq: jacobian_of_a_minor}) is zero again, because the row indexed by $j$ appears twice.

Finally, we consider the rows of the augmented Jacobian corresponding to constraints that $x_{kl}=0$.
The Jacobian of $x_{kl}$ consists of the entry $(k,l)$ being equal to one and all other entries being equal to zero.
If $i \neq k$ and $j \neq k$, then the Jacobian of this constraint  multiplied by $X^{i \leftrightarrow j}$ is clearly zero.
Otherwise the Jacobian of $x_{kl}$ multiplied by $X^{i \leftrightarrow j}$ is either $x_{il}$ of $x_{jl}$.
However both $x_{il}=x_{jl}=0$ by the assumption $i\sim_R^S j$.

This proves that $r_{U}^{(i,j)}\in\mathcal{I}(Z_{U,r}^{S})$ for all $i\sim_R^S j$. Similarly, one verifies that $c_{U}^{(i,j)}\in\mathcal{I}(Z_{U,r}^{S})$ for all $i\sim_C^S j$ by applying the same argument with the vector
\[
X_{i \leftrightarrow j}\coloneqq
\begin{bmatrix}
-x_{1j} & 0 & \cdots & 0 & x_{1i} & 0 & \cdots & 0 & -x_{nj} & 0 & \cdots & 0 & x_{ni} & 0 & \cdots & 0 \\
\end{bmatrix}^T\!.\vspace{-10pt}
\]
\qedhere
\end{proof}

\begin{conjecture}\label{conj: linear span}
Let $S\subset[m]\times[n]$, $U\in\R^{m\times n}$ and $r\in [m-1]$.
Then $\langle Z_{U,r}^{S}\rangle=H_U^{S}$ if and only if $\mathcal{L}_r^S$ is irreducible.
\end{conjecture}

\begin{example}[$m=n=3$, $r=2$, $S=\{(1,1)\}$]\label{example: 3x3 one zero rank two}
By experimental computation, we observe that $\mathrm{EDdegree}(\mathcal{L}_{2}^S)=8$ (see Table \ref{table: EDdegree 1 zero}).
If linearly independent, the eight critical points should span an eight-dimensional linear space $\langle Z_{U,2}^{S}\rangle\subset\C^{3\times 3}$.
We verified symbolically using Gr\"{o}bner bases and elimination that $\mathcal{I}(\langle Z_{U,2}^{S}\rangle)=\langle r_{U}^{(2,3)},c_{U}^{(2,3)},x_{11}\rangle$, thus confirming Conjecture \ref{conj: linear span}.
\end{example}

We will show in Examples~\ref{example: 3x3 one zero basis} and~\ref{example: 3x3 two zeros row basis} that if $\mathcal{L}_r^S$ is not irreducible, then $\langle Z_{U,r}^{S}\rangle$ can be strictly contained in $H_U^{S}$. Example~\ref{example: 3x3 one zero basis} is for structured rank-one approximation.
We know from Proposition \ref{prop: ED degree rank 1 zero pattern} that $\mathcal{L}_1^S$ is never irreducible if $S\neq\emptyset$, and hence Conjecture \ref{conj: linear span} suggests that in the structured setting $\langle Z_{U,1}^{S}\rangle$ is always strictly contained in $H_U^{S}$. Example~\ref{example: 3x3 two zeros row basis} is for rank-two approximation.

\begin{example}[$m=n=3$, $r=1$, $S=\{(1,1)\}$]\label{example: 3x3 one zero basis}
In this case, the variety $\mathcal{L}_1^S$ has two irreducible components corresponding to the minimal coverings $(\{1\},\emptyset)$ and $(\emptyset,\{1\})$.
Moreover $\mathrm{EDdegree}(\mathcal{L}_1^S)=4$ by Corollary \ref{corol: ED degree rank 1 diagonal}.
The four critical points on $\mathcal{L}_1^S$ are obtained in this way:
\begin{itemize}
    \item[$(i)$] by computing the SVD of the $3\times 3$ matrix having zero first row and coinciding with $U$ elsewhere (two critical points $C_1,C_2$),
    \item[$(ii)$] by computing the SVD of the $3\times 3$ matrix having zero first column and coinciding with $U$ elsewhere (two critical points $C_3,C_4$).
\end{itemize}
One verifies immediately that the critical space $H_U^{S}$ is six-dimensional.
Therefore $\langle Z_{U,1}^{S}\rangle$ is strictly contained in $H_U^{S}$ and motivates our hypothesis in Conjecture \ref{conj: linear span}.
\end{example}

A higher rank example is showed below.

\begin{example}[$m=n=3$, $r=2$, $S=\{(1,1),(1,2)\}$]\label{example: 3x3 two zeros row basis}
The determinant of a $3\times 3$ matrix $X=(x_{ij})$ with zero pattern $S$ is $\det(X)=x_{13}(x_{21}x_{32}-x_{31}x_{22})$.
Then the variety $\mathcal{L}_{2}^S$ has two components $\mathcal{V}_1=\mathcal{V}(x_{11},x_{12},x_{13})$, $\mathcal{V}_2=\mathcal{V}(x_{11},x_{12},x_{21}x_{32}-x_{31}x_{22})$ and by Lemma \ref{lemma: ED degree union}
\[
\mathrm{EDdegree}(\mathcal{L}_{2}^S)=\mathrm{EDdegree}(\mathcal{V}_1)+\mathrm{EDdegree}(\mathcal{V}_2)=1+2=3\,.
\]
The critical point on $\mathcal{V}_1$ is the projection of $U$ onto $\mathcal{V}_1$.
The two critical points on $\mathcal{V}_2$ come by projecting the third column of $U$ and computing the SVD of the non-zero $2\times 2$ block.
In particular, the linear span $\langle Z_{U,2}^{S}\rangle$ is three-dimensional, whereas the critical space $H_U^{S}$ has dimension five.
\end{example}

\begin{corollary}\label{corol: codim linear span}
Let $S\subset[m]\times[n]$. If $U$ is sufficiently generic, then the codimension of $H_U^{S}\subset\C^{m\times n}$ is
\[
\mathrm{codim}(H_U^{S})=\sum_{C\in [m]/\sim_R^S }\binom{|C|}{2}+\sum_{D\in [n]/\sim_C^S }\gamma_D+|S|,
\]
where
\[
\gamma_D=
\begin{cases}
\binom{|D|}{2} & \mbox{if $|D|\le m$}\\
m(|D|-1)-\binom{m}{2} & \mbox{if $|D|\ge m$}\,.
\end{cases}
\]
\end{corollary}

\begin{proof}[Proof of Corollary~\ref{corol: codim linear span}]
We have to show how many of all the linear polynomials appearing in \eqref{eq: linear relations constrained case} are linearly independent.

The first two sets of generators in \eqref{eq: linear relations constrained case} are precisely of the type explained in Remark \ref{rmk: on linear relations with rows or columns}.
On one hand, since $m\le n$, then $|C|\le n$ for every equivalence class $C\in [m]/\!\!\sim_R^S $.
This means that all relations $r_{U}^{(i,j)}$, where $i,j\in C$, are linearly independent.
By Remark \ref{rmk: on linear relations with rows or columns} and since equivalence classes on rows are disjoint, this gives in total the first $\sum_{C\in [m]/\sim_R^S }\binom{|C|}{2}$ independent conditions.
On the other hand, if $D\in[n]/\!\!\sim_C^S $ and $|D|\le m$, then again by Remark \ref{rmk: on linear relations with rows or columns} all relations $c_{U}^{(i,j)}$, where $i,j\in D$, are linearly independent, thus giving $\binom{|D|}{2}$ independent conditions.
Otherwise if $|D|\ge m$, then $m(|D|-1)-\binom{m}{2}$ among the last equations are linearly independent.
Since equivalence classes on columns are disjoint, this gives in total the second $\sum_{D\in [n]/\sim_C^S }\gamma_D$ independent conditions.
Moreover, again by Remark \ref{rmk: on linear relations with rows or columns}, each condition on rows is not a linear combination of equations involving columns, and vice versa.

Finally, consider the last $|S|$ conditions coming from the zero pattern.
Trivially each relation $x_{ij}$ is independent from the other variables $x_{rs}$ with $(r,s)\in S$.
Moreover, all the conditions $\{x_{ij}\mid(i,j)\in S\}$ are independent from the first two sets of equations because the first two contain no variables with indices in $S$.
\end{proof}

If Conjecture \ref{conj: linear span} is true, then the statement in Corollary~\ref{corol: codim linear span} holds for $\langle Z_{U,r}^{S}\rangle$ when $\mathcal{L}_r^S$ is irreducible.

\begin{remark}\label{rmk: coefficients}
Consider again the situation of Example \ref{example: 3x3 one zero rank two} with $m=n=3$, $r=2$, $S=\{(1,1)\}$.
We recall that in this case $\mathcal{I}(\langle Z_{U,2}^{S}\rangle)=\langle r_{U}^{(2,3)},c_{U}^{(2,3)},x_{11}\rangle$, thus confirming Conjecture \ref{conj: linear span}.
In particular $\dim(\langle Z_{U,2}^{S}\rangle)=\dim(H_U^{S})=6$.
One might try to extend the basis $\{C_1,C_2,C_3,C_4\}$ of $\langle Z_{U,1}^{S}\rangle$ given in Example \ref{example: 3x3 one zero basis} to form a basis of $\langle Z_{U,2}^{S}\rangle$, in the most ``natural'' way.
In this example, we consider the additional two rank-one matrices $C_5,C_6$ obtained by computing the SVD of the $3\times 3$ matrix having zero first row and column and coinciding with $U$ elsewhere.
One might check that the six rank-one matrices $C_1,\ldots,C_6$ are linearly independent and form a basis of $H_U^{S}$.

The matrices $C_5$ and $C_6$ are ``good'' in the sense that they are computed directly from the data matrix $U$ via projections and SVDs.
Any critical point $X\in\mathcal{L}_{2}^S$ may be written as $X=\alpha_1C_1+\cdots+\alpha_6C_6$ for some complex coefficients $\alpha_1,\ldots,\alpha_6$.
In Table \ref{table: coeffs 3x3 one zero} we display the coefficients $\alpha_i$ of the eight critical points on $\mathcal{L}_{2}^S$ with respect to the data matrix
\[
U =
\begin{pmatrix}
78.57 & 93.47 & 51.33\\
-58.54 & -7.64 & 34.34\\
53.53 & -89.96 & -87.14
\end{pmatrix}\,.
\]
\begingroup
\setlength{\tabcolsep}{4pt}
\renewcommand{\arraystretch}{1}

\begin{table}[htbp]
\centering
\begin{tabular}{|c||c|c|c|c|c|c|}
\hline
&$\alpha_1$&$\alpha_2$&$\alpha_3$&$\alpha_4$&$\alpha_5$&$\alpha_6$\\
\hhline{|=||======|}
$X_1$&0.96&-0.11&0.01&-0.27&0.04&1.06\\
$X_2$&0.02&1.05&0.00&1.13&-0.03&-1.06\\
$X_3$&0.04&-0.12&0.99&-0.29&-0.01&0.77\\
$X_4$&1.00&0.99&1.00&0.97&-1.00&-0.93\\
$X_5$&0.77-0.87\,$i$&-1.39-0.41\,$i$&0.73-0.30\,$i$&-2.67-0.72\,$i$&-0.28+1.31\,$i$&3.75+3.41\,$i$\\
$X_6$&0.77+0.87\,$i$&-1.39+0.41\,$i$&0.73+0.30\,$i$&-2.67+0.72\,$i$&-0.28-1.31\,$i$&3.75-3.41\,$i$\\
$X_7$&2.45-0.71\,$i$&-0.86-3.45\,$i$&1.65-0.37\,$i$&2.81+5.93\,$i$&-3.59+1.46\,$i$&3.73-4.66\,$i$\\
$X_8$&2.45+0.71\,$i$&-0.86+3.45\,$i$&1.65+0.37\,$i$&2.81-5.93\,$i$&-3.59-1.46\,$i$&3.73+4.66\,$i$\\
\hline
\end{tabular}
\vspace*{1mm}
\caption{The critical points $X_i\in\mathcal{L}_{2}^S\subset\C^{3\times 3}$ of $d_U$ in the basis $\{C_1,\ldots,C_6\}$.}\label{table: coeffs 3x3 one zero}
\end{table}
\endgroup

More generally, knowing the ideal of critical points $X\in\mathcal{L}_{r}^S$ for $d_U$ and a basis $\{C_1,\ldots,C_k\}$ of $\langle Z_{U,r}^{S}\rangle$ whose elements depend only on $U$, allows to compute the ideal $\mathcal{J}_\alpha\subset\C[\alpha_1,\ldots,\alpha_k]$ of relations among the coefficients $\alpha_i$ of a representation of $X$.

This idea needs further investigation and is motivated by the unstructured case.
Indeed, given $U\in\R^{m\times n}$, the critical points $C_1,\dots,C_m$ on $\mathcal{X}_1$ of $d_U$ computed from the SVD of $U$ form a basis of $\langle Z_{U,r}^{S}\rangle=H_U^{S}$ for any $r\in[m]$.
By Theorem \ref{thm: Eckart-Young}, the critical points on $\mathcal{X}_r$ are written uniquely as $X=\alpha_1C_1+\cdots+\alpha_mC_m$ for some coefficients $\alpha_j\in\{0,1\}$.
In this case, the ideal $\mathcal{J}_\alpha$ is zero-dimensional in $\C[\alpha_1,\ldots,\alpha_m]$ and its degree is equal to $\mathrm{EDdegree}(\mathcal{X}_r)=\binom{m}{r}$.
\end{remark}

\subsection{Linear relations among critical points for generic subspaces} \label{sec: general subspaces}

The statement of Theorem \ref{thm: linear_space_for_rank_r_critical_points} can be adapted to arbitrary linear subspaces $\mathcal{L}$ of $\C^{m\times n}$.

\begin{proposition}
Let $U\in\R^{m\times n}$.
Let $\mathcal{L}\subset\C^{m\times n}$ be the linear subspace defined by the linear forms $L_1,\ldots,L_s$ and let $Z_{U,r}^\mathcal{L}$ be the set of critical points of $d_U$ on the variety $\mathcal{L}_r$.
For all $i<j$, $i,j\in[m]$ and $k\in[s]$, if $\nabla L_k\cdot X^{i \leftrightarrow j}\in\langle L_1,\ldots,L_s\rangle$, then $r_U^{(i,j)}\in\mathcal{I}(Z_{U,r}^\mathcal{L})$.
Similarly, for all $i<j$, $i,j\in[n]$ and $k\in[s]$, if $\nabla L_k\cdot X_{i \leftrightarrow j}\in\langle L_1,\ldots,L_s\rangle$, then $c_U^{(i,j)}\in\mathcal{I}(Z_{U,r}^\mathcal{L})$.
\end{proposition}

Symbolic computations in Macaulay2 motivate the following question for generic subspaces.

\begin{question}\label{question:linear_span_general_linear_subspace}
Let us replace $\mathcal{L}^S$ with a non-zero proper subspace $\mathcal{L}\subset\C^{m\times n}$. Is it true that $\mathcal{L}=\langle Z_{U,r}^\mathcal{L}\rangle$ for every $r\in [m-1]$ if $\mathcal{L}$ is generic?
\end{question}

We have checked symbolically with a Macaulay2 code that $\mathcal{L}=\langle Z_{U,r}^\mathcal{L}\rangle$ when $\mathcal{L}$ is a generic subspace of codimension $\mathrm{codim}(\mathcal{L})\in\{1,2\}$ and for small formats.
This suggests that the answer to Question~\ref{question:linear_span_general_linear_subspace} is positive.
As showed in Theorem~\ref{thm: linear_space_for_rank_r_critical_points}, the same is not true for $\mathcal{L}=\mathcal{L}^S$ for some zero pattern $S$.

\subsection{The affine span of critical points} \label{sec: affine relations}

In the previous sections, we regarded structured critical points on $\mathcal{L}_{r}^S$ as vectors in $\C^{m\times n}$ and studied all linear relations among them.
In this section, we look for their affine span.
This is important to understand how the critical points of $d_U$ are distributed inside the critical space $H_U^{S}$.
The affine span in the unstructured case is characterized in Proposition~\ref{prop:affine_relations} and Corollary~\ref{corol: affine linear relations unconstrained}. Conjecture~\ref{conj: affine relations} suggests affine relations satisfied by the rank-$r$ critical points in the structured setting.

\begin{proposition} \label{prop:affine_relations}
Assume $m=n$ and $S=\emptyset$.
Given $U\in\R^{m\times m}$ of rank $k$, for every $r\in[k]$ the $\binom{k}{r}$ critical points of $d_U$ on $\mathcal{X}_r$ span an affine hyperplane $W_{U,r}^S\subset H_U^{S}$ of equation
\[
W_{U,r}^S\colon\langle X,C(U)\rangle_F-r\det(U)=0\,,
\]
where $C(U)$ is the cofactor matrix of $U$ whose $(i,j)$-th entry is $(-1)^{i+j}M_{[m]\setminus\{i\},[m]\setminus\{j\}}(U)$.

\noindent In particular, if $\det(U)\neq 0$, then $\{W_{U,r}^S\}_{r\in[m]}$ is a finite family of parallel hyperplanes contained in $H_U^{S}$.
Otherwise if $\det(U)=0$, then $W_{U,r}^S$ is a linear hyperplane whose equation does not depend on the rank $r$.
\end{proposition}

\begin{proof}
Let $U\in\R^{m\times m}$ of rank $k$, written in SVD form as $U=A\Sigma B^T$.
One verifies immediately that $C(U)=AC(\Sigma)B^T$ and that the $j$-th diagonal entry of $C(\Sigma)$ is equal to $\prod_{l\neq j}\sigma_l$ for all $j\in[m]$.
Given $r\in[k]$, by Theorem \ref{thm: Eckart-Young} a critical point $X\in\mathcal{X}_r$ of $d_U$ is of the form $X=A\Sigma_I B^T$ for some $I=\{i_1<\cdots<i_r\}\subset[m]$, where $\Sigma_I=\mathrm{diag}(0,\ldots,\sigma_{i_1},\ldots,\sigma_{i_r},\ldots,0)$.
Without loss of generality, assume $I=[r]\subset[m]$. We denote by $a_1,\ldots,a_m$ and $b_1,\ldots,b_m$ the orthonormal columns of $A$ and $B$, respectively. Then
\begin{align*}
\langle X, C(U)\rangle_F & = \left\langle A\Sigma_{[r]}B^T, AC(\Sigma)B^T\right\rangle_F\\
& = \left\langle\sum_{i=1}^r\sigma_ia_ib_i^T,\sum_{j=1}^m\left(\prod_{l\neq j}\sigma_l\right)a_jb_j^T\right\rangle_F\\
& = \sum_{i=1}^r\sum_{j=1}^m\sigma_i\left(\prod_{l\neq j}\sigma_l\right)\left\langle a_ib_i^T,a_jb_j^T\right\rangle_F\\
& = \sum_{i=1}^r\sigma_i\left(\prod_{l\neq i}\sigma_l\right) = r\det(\Sigma) = r\det(U)\,,
\end{align*}
where we have applied the identity ($\delta_{ij}$ is the Kronecker delta)
\[
\left\langle a_ib_i^T,a_jb_j^T\right\rangle_F=\mathrm{trace}\left(a_ib_i^Tb_ja_j^T\right)=\delta_{ij}\mathrm{trace}\left(a_ia_j^T\right)=\delta_{ij}\,a_i^Ta_j=\delta_{ij}\,.\qedhere
\]
\end{proof}

The following corollary generalizes Proposition~\ref{prop:affine_relations} to non-squared case.

\begin{corollary}\label{corol: affine linear relations unconstrained}
Assume $m\le n$ and $S=\emptyset$.
Given $U\in\R^{m\times n}$ of rank $k$, for every $r\in[k]$ the $\binom{k}{r}$ critical points of $d_U$ on $\mathcal{X}_r$ span an affine hyperplane $W_{U,r}^S\subset H_U^{S}$ of equation
\[
W_{U,r}^S\colon\left\langle X_{[m],I},C(U_{[m],I})\right\rangle_F-r\det(U_{[m],I})=0\,,
\]
where $I=\{i_1<\cdots<i_m\}\subset[n]$ and, modulo the ideal $\mathcal{I}(H_U^{S})$, the above equation does not depend on the particular choice of $I$.
\end{corollary}

In the next example we observe which of the affine relations of Corollary \ref{corol: affine linear relations unconstrained} still hold true in the structured case.

\begin{example}\label{example: 3x4 one zero affine relation}
Let $U\in\R^{3\times 4}$.
In this example we investigate the best rank-two approximation of $U$ with zero pattern $S=\{(1,1)\}$.
We verified symbolically that $\mathrm{EDdegree}(\mathcal{L}_{r}^S)=8$.
The eight critical points of $d_U$ span a seven-dimensional linear space $\langle Z_{U,2}^{S}\rangle=H_U^{S}$ whose ideal is
\[
\mathcal{I}(H_U^{S})=\langle r_U^{(2,3)},c_U^{(2,3)},c_U^{(2,4)},c_U^{(3,4)},x_{11}\rangle\,.
\]
Moreover, the critical points satisfy the affine relation in $H_U^{S}$
\[
\left\langle X_{[3],I},C(U_{[3],I})\right\rangle_F-2\det(U_{[3],I})=0\ \text{for}\ I=\{2,3,4\}\,.
\]
\end{example}

The previous example motivates the following conjecture in structured setting.

\begin{conjecture}\label{conj: affine relations}
Assume that $\mathcal{L}_r^S$ is irreducible. Given a subset $I\subset[n]$ with $|I|=m$, the complex critical points of $d_U$ on $\mathcal{L}_r^S$ satisfy the additional affine relation
\begin{equation}\label{eq: affine relation constrained case}
\left\langle X_{[m],I},C(U_{[m],I})\right\rangle_F-r\det(U_{[m],I})=0
\end{equation}
if any only if $S\cap([m]\times I)=\emptyset$.
\end{conjecture}

Conjecture \ref{conj: affine relations} is not valid if $\mathcal{L}_r^S$ is not irreducible, as showed in the next example.

\begin{example}[$m=n=3$, $r=2$, $S=\{(1,1),(1,2)\}$]\label{example: 3x3 two zeros row basis, affine}
Following up Example \ref{example: 3x3 two zeros row basis}, we observe that the ideal of affine relations among the critical points is
\begin{equation}\label{eq: affine relation 3x3 two zeros row}
\left\langle r_U^{(2,3)},c_U^{(1,2)},x_{11},x_{12},x_{23}-u_{23},x_{33}-u_{33},\left\langle\widetilde{X},C(\widetilde{U})\right\rangle_F-2\det(\widetilde{U})\right\rangle\,,
\end{equation}
where the matrices $\widetilde{U}$ and $\widetilde{X}$ coincide with $U$ and $X$ outside $S$, respectively, and are zero otherwise.
The seven affine relations are independent and thus define an affine plane in $\C^{3\times 3}$.

The first four relations are linear and define the critical space $H_U^{S}$.
The last affine relation in \eqref{eq: affine relation 3x3 two zeros row} is equivalent to $\left\langle X,C(U)\right\rangle_F-r\det(U)=0$ modulo $\mathcal{I}(\mathcal{L}^S)=\langle x_{11},x_{12}\rangle$, since by Lemma \ref{lemma: critical points projection} the matrix $U$ shares the same critical points of its projection $\pi_{\mathcal{L}^S}(U)$ onto $\mathcal{L}^S$.
On one hand, this affine relation coincides with \eqref{eq: affine relation constrained case} for $I=J=[3]$.
On the other hand, in this case $S\cap([3]\times[3])\neq\emptyset$.
\end{example}

\subsection{Special nonlinear relations among structured critical points}\label{sec: special nonlinear relations}

In Section \ref{sec:linear_relations}, we observed that the special linear constraints $\mathcal{L}^S$ coming from zero patterns $S\subset[m]\times[n]$ preserve some of the linear relations among unstructured critical points.
In this subsection, we deal with special nonlinear relations among unstructured or structured critical points. The following is a consequence of \cite[Theorem 5.2]{draisma2016euclidean} and Theorem \ref{thm: Eckart-Young}.

\begin{proposition}\label{lemma: nonlinear relations minors unconstrained}
Let $U\in\R^{m\times n}$.
Every critical point $X\in \mathcal{X}_r$ of $d_U$ is such that $U-X\in \mathcal{X}_r^\vee = \mathcal{X}_{m-r}$ and $U-X$ is a critical point of $d_U$ as well.
In particular, for every subset $A\subset[m]$ and $B\subset[n]$ with $|A|=|B|\ge m-r+1$, we have that
\begin{equation}\label{eq: det(U-X)AB}
M_{A,B}(U-X)=0\,.
\end{equation}
\end{proposition}

The main result of this subsection is the next proposition which states that some of the relations in \eqref{eq: det(U-X)AB} hold even in the structured case.

\begin{proposition}\label{prop: nonlinear relations minors constrained}
Let $S\subset[m]\times[n]$, $U\in\R^{m\times n}$ and consider a critical point $X=(x_{ij})\in\mathcal{L}_{r}^S$ of $d_U$.
Let $A\times B\subset [m]\times[n]$ with $|A|=|B|\ge m-r+1$ and such that $S\cap(A\times B)=\emptyset$.
Then $M_{A,B}(U-X)=0$.
\end{proposition}

\begin{proof}
In the following, we denote by $L_1,\ldots,L_{|S|}$ the constraints that set the entries in $S\subset[m]\times[n]$ of the structured matrices to be equal to zero.
We recall from Theorem~\ref{thm: linear_space_for_rank_r_critical_points} that the critical points $X=(x_{ij})\in \mathcal{L}_{r}^S$ of $d_U$ are the solutions of the system
\begin{equation}\label{eq: polynomial system with Lagrange multipliers}
  \begin{cases}
    M_{I,J}(X) = 0 & \mbox{$\forall\,|I|=|J|=r+1$}\\
    L_k(X) = 0 & \forall\,k\in [|S|]\\
    u_{ij} - x_{ij} = \sum_{I,J}\frac{\partial M_{I,J}}{\partial x_{ij}}\lambda_{I,J} + \sum_{k=1}^t\frac{\partial L_k}{\partial x_{ij}} \mu_k & \forall\,(i,j)\in[m]\times [n]\,.
  \end{cases}
\end{equation}

Let $A\times B\subset[m]\times[n]$ with $|A|=|B|=s$ and assume $S\cap(A\times B)=\emptyset$.
Equivalently we have $\frac{\partial L_k}{\partial x_{ij}}=0$ for all $k\in[t]$ and for all $(i,j)\in A\times B$.
Define the matrix
\[
\partial(f)\coloneqq\left(\frac{\partial f}{\partial x_{ij}}\right)\in\C^{m\times n}\quad\forall\,f=f(X)\in\C[X]\,.
\]
Using the third set of equations in \eqref{eq: polynomial system with Lagrange multipliers}, we get the identity
\begin{equation}\label{eqn:FAB}
M_{A,B}(U-X) = M_{A,B}\left(\sum_{I,J}\partial(M_{I,J}(X))\lambda_{I,J}\right)\eqqcolon F_{A,B}.
\end{equation}

Under the assumption $S\cap(A\times B)=\emptyset$, the polynomial $F_{A,B}$ does not depend on the linear constraints $L_k(X)=0$, and thus it is independent of $S$. Hence the equality~(\ref{eqn:FAB}) for $S=\emptyset$ involves the same $F_{A,B}$ as for any other $S$ satisfying $S\cap(A\times B)=\emptyset$. By Proposition~\ref{lemma: nonlinear relations minors unconstrained}, $M_{A,B}(U-X)=0$ in the unstructured case, and hence  $M_{A,B}(U-X)=0$ in the structured case.
\end{proof}

\begin{remark}\label{lemma: FAB with minors}
The polynomial $F_{A,B}$ is homogeneous in the $\ell_{I,J}$'s.
The coefficients of the monomials of $F_{A,B}$ in the $\ell_{I,J}$'s are polynomials in $\C[X]$. In particular, they belong to the ideal $\mathcal{I}(\mathcal{X}_r)$ for all $A\times B\subset[m]\times[n]$ with $s\ge m-r+1$.
\end{remark}

\begin{remark}
The condition $S\cap(A\times B)=\emptyset$ in Proposition \ref{prop: nonlinear relations minors constrained} is sufficient but not necessary to prove that $M_{A,B}(U-X)=0$.
For example, let $m=n=3$, $r=2$, $s=1$ and $S=\{(1,1)\}$.
If $A=B=[3]$, we obtain that (here $\lambda_{[3],[3]}=\lambda$, $\mu_1=\mu$ and $M_{[3],[3]}(X)=\det(X)$)
\[
F_{[3],[3]}=\det(X)^2\lambda^3+x_{11}\det(X)\lambda^2\mu\,,
\]
that is, $F_{[3],[3]}\in\mathcal{I}(\mathcal{X}_2)$ and consequently $\det(U-X)=0$.
\end{remark}

The statement of Proposition \ref{prop: nonlinear relations minors constrained} can be generalized to arbitrary linear sections $\mathcal{L}_r$ of $\mathcal{X}_r$.
The condition which replaces $S\cap(A\times B)=\emptyset$ is simply that $\frac{\partial L}{\partial x_{ij}}=0$ for all $L\in\mathcal{I}(\mathcal{L})$, namely no linear constraint depends by variables $x_{ij}$ with indices in $A\times B$.
Again this condition is far from being necessary.
To show this, below we restrict to one constraint $L=\sum_{i,j}v_{ij}x_{ij}$ and to the variety of corank one square matrices $\mathcal{X}_{m-1}\subset\C^{m\times m}$.
We prove that the rank of $U-X$ is completely characterized by the rank of the coefficient matrix $V=(v_{ij})$.

\begin{proposition}\label{prop: conditions L on rank(U-X)}
Consider a linear form $L=\sum_{i,j=1}^mv_{ij}x_{ij}$ for some matrix $V=(v_{ij})\in\C^{m\times m}$.
Let $U\in\R^{m\times m}$ and let $X\in\mathcal{L}_{m-1}=\mathcal{X}_{m-1}\cap\mathcal{V}(L)$ be a critical point of $d_U$.
Then for all $2\le k\le m$
\[
\mathrm{rk}(U-X)\le k-1\text{ if and only if }\rank(V)\le k-2\,.
\]
In particular $\rank(U-X)\ge 1$ if $V\neq 0$ and $\rank(U-X)=1$ if and only if $V=0$, namely in the unstructured case.
\end{proposition}

\begin{proof}
Recall the notation introduced at the beginning of Section \ref{sec: special relations}.\\
We have $\mathcal{I}(\mathcal{L}_{m-1})=\langle\det(X),L\rangle$.
Suppose that $\mathrm{rk}(U-X)\le k-1$, namely $M_{A,B}(U-X)=0$ for all $A,B\subset[m]$ with $|A|=|B|=k$. Using the system \eqref{eq: polynomial system with Lagrange multipliers} we get that
\begin{equation}\label{eq: def GAB}
M_{A,B}(U-X)=M_{A,B}(\lambda\partial(\det(X))+\mu\partial(L))=M_{A,B}(\lambda C(X)+\mu V)\,,
\end{equation}
where $C(X)=(C_{ij}(X))$ is the cofactor matrix of $X$.
Assume $A=B=[k]$. Our goal is to expand the polynomial at the right-hand side of \eqref{eq: def GAB}, which we call $G_k$ for brevity.
First, we consider the following expansion of $M_{[k],[k]}(P+Q)$ for all $P,Q\in\C^{m\times m}$:
\[
M_{[k],[k]}(P+Q)=\sum_{\substack{I,J\subset[k]\\|I|=|J|}}(-1)^{I+J}M_{I,J}(P)M_{[k]\setminus I,[k]\setminus J}(Q)\,,
\]
where $(-1)^{I+J}=(-1)^{\sum_{i\in I}i+\sum_{j\in J}j}$. This identity follows from the Laplace expansion of the determinant in multiple columns, that can be found for example in \cite[Lemma $A.1(f)$]{CSS}. We apply it in the case $P=\lambda C(X)$ and $Q=\mu V$:
\begin{align*}
G_k & = \sum_{\substack{I,J\subset[k]\\|I|=|J|}}(-1)^{I+J}M_{I,J}(C(X))M_{[k]\setminus I,[k]\setminus J}(V)\,\lambda^{|I|}\mu^{k-|I|}\,.
\end{align*}
Then we apply the identity
\[
M_{I,J}(C(X))=
\begin{cases}
1 & \text{if $I=J=\emptyset$}\\
(-1)^{I+J}\det(X)^{|I|-1}M_{[k]\setminus I,[k]\setminus J}(X) & \text{if $|I|=|J|\ge 1$}
\end{cases}
\]
which follows from the relation $C(X)=\det(X)\,X^{-T}$ and the {\em Jacobi complementary minor Theorem}, see \cite{Lalonde} and \cite[Lemma $A.1(e)$]{CSS}. Hence we get
\[
G_k = M_{[k],[k]}(V)\mu^k + \sum_{\substack{I,J\subset[k]\\|I|=|J|\ge 1}}\det(X)^{|I|-1}M_{[k]\setminus I,[k]\setminus J}(X)M_{[k]\setminus I,[k]\setminus J}(V)\,\lambda^{|I|}\mu^{k-|I|}\,.
\]
Finally, using the condition $\det(X)=0$, the previous identity simplifies to
\[
G_k = M_{[k],[k]}(C)\mu^k + \sum_{i,j\in[k]}M_{[k]\setminus\{i\},[k]\setminus \{j\}}(X)M_{[k]\setminus\{i\},[k]\setminus\{j\}}(V)\,\lambda\mu^{k-1}\,.
\]
Similarly, for arbitrary $A,B\subset[m]$ with $|A|=|B|=k$, we get that
\[
M_{A,B}(U-X) = M_{A,B}(C)\mu^k + \sum_{i\in A,j\in B}M_{A\setminus\{i\},B\setminus \{j\}}(X)M_{A\setminus\{i\},B\setminus\{j\}}(V)\,\lambda\mu^{k-1}\,.
\]
The consequence is that a critical point $X$ of $d_U$ is such that $\rank(U-X)\le k-1$ if and only if $M_{A\setminus\{i\},B\setminus\{j\}}(V)=0$ for all $i\in A$, $j\in B$ and $A,B\subset[m]$ with $|A|=|B|=k$, or equivalently $\rank(V)\le k-2$.
\end{proof}

The main observation coming from Proposition \ref{prop: conditions L on rank(U-X)} is that a generic linear constraint $L$ destroys the structure of critical points coming from Theorem \ref{thm: Eckart-Young}, in particular the relations $M_{A,B}(U-X)$ for suitable $A$ and $B$.
However, if $L$ is special, these conditions might still hold, even in the case when $L$ involves entries $x_{ij}$ with $(i,j)\in A\times B$.

\section{Computations of Euclidean Distance degrees}\label{sec: EDdegree}

In this section we present various experiments that study the ED degree of $\mathcal{L}_r^S$, when $r\ge 2$ and the zero pattern $S$ involves only elements in the diagonal.

\medskip
First, we restrict to square matrices and consider the zero pattern $S=\{(1,1)\}$.
Since the number of (complex) critical points of $d_U$ on $\mathcal{L}_{n-1}^S$ is constant for a generic (complex) data matrix $U$, it is reasonable to apply a monodromy technique for computing these critical points numerically.
For this, we use the HomotopyContinuation.jl~\cite{HomotopyContinuation.jl} software package.
The number of solutions obtained (that is, the ED degree of $\mathcal{L}_{n-1}^S$ with respect to the Frobenius inner product) is reported in Table~\ref{table: EDdegree 1 zero}.
We checked symbolically in Maple\texttrademark~2016 \cite{maple} that the number of numerical solutions obtained coincides with the degree of the ideal of rank-$r$ critical points $\mathcal{I}(Z_{U,n-1})$, that is the ED degree of $\mathcal{L}_{n-1}^S$.

\begingroup
\begin{table}[htbp]
	\centering
	\begin{tabular}{|c||cccccccc|}
		\hline
		$n$ & 3 & 4 & 5 & 6 & 7 & 8 & 9 & 10\\
    	\hhline{|=||========|}
		$\mathrm{EDdegree}(\mathcal{L}_{n-1}^S)$ & 8 & 13 & 18 & 23 & 28 & 33 & 38 & 43 \\
		\hline
	\end{tabular}
	\vspace*{1mm}
	\caption{ED degrees for $n \times n$ matrices of $\rank \leq n-1$ and $S=\{(1,1)\}$.}
	\label{table: EDdegree 1 zero}
\end{table}
\endgroup

Our experimental results support the following conjecture.
\begin{conjecture}\label{conj:ED-1z-rn-1}
Consider the variety $\mathcal{L}_{n-1}^S\subset\C^{n\times n}$, where $S=\{(1,1)\}$.
Then
\[
\mathrm{EDdegree}(\mathcal{L}_{n-1}^S)=5(n-1)-2\,.
\]
\end{conjecture}

Next, we fix the diagonal zero pattern $S = \{(1,1),\ldots,(s,s)\}$ for $s\in[4]$, and we consider the variety $\mathcal{L}_{2}^S\subset\C^{m\times n}$.
We present in Tables \ref{table: 1z-r2},\ref{table: 2z-r2},\ref{table: 3z-r2},\ref{table: 4z-r2} the values of $\mathrm{EDdegree}(\mathcal{L}_{2}^S)$ computed depending on the format $m\times n$.
Our experiments support the following conjectural formulas.

\begin{conjecture}\label{conj: ED-1,2,3,4z-r2}
Consider the variety $\mathcal{L}_{2}^S\subset\C^{m\times n}$ with $S = \{(1,1),\ldots,(s,s)\}$ for $s\in[4]$. Let $l=\min(m,n)$. Then
\[
\mathrm{EDdegree}(\mathcal{L}_{2}^S)=
\begin{cases}
3(l-1)^2-2(l-1) & \text{if $s=1$}\\
18(l-2)^2+6(l-2)+1 & \text{if $s=2$ and $m\neq n$}\\
18(l-2)^2+10(l-2)+1 & \text{if $s=2$ and $m=n$}\\
108(l-3)^2+144(l-3)+30 & \text{if $s=3$ and $m=n$}\\
648(l-4)^2+1600(l-4)+488 & \text{if $s=4$ and $m=n$}\,.
\end{cases}
\]
\end{conjecture}

\begingroup
\begin{table}[htbp]
\centering
\begin{tabular}{|c||cccccccccc|}
\hline
\backslashbox{$s$}{$n$} & 3 & 4 & 5 & 6 & 7 & 8 & 9 & 10 & 11 & 12\\
\hhline{|=||==========|}
3 & 8 & 8 & 8 & 8 & 8 & 8 & 8 & 8 & 8 & 8 \\
4 & 8 & 21 & 21 & 21 & 21 & 21 & 21 & 21 & 21 & 21 \\
5 & 8 & 21 & 40 & 40 & 40 & 40 & 40 & 40 & 40 & 40 \\
6 & 8 & 21 & 40 & 65 & 65 & 65 & 65 & 65 & 65 & 65 \\
7 & 8 & 21 & 40 & 65 & 96 & 96 & 96 & 96 & 96 & 96 \\
8 & 8 & 21 & 40 & 65 & 96 & 133 & 133 & 133 & 133 & 133 \\
9 & 8 & 21 & 40 & 65 & 96 & 133 & 176 & 176 & 176 & 176 \\
10 & 8 & 21 & 40 & 65 & 96 & 133 & 176 & 225 & 225 & 225 \\
11 & 8 & 21 & 40 & 65 & 96 & 133 & 176 & 225 & 280 & 280 \\
12 & 8 & 21 & 40 & 65 & 96 & 133 & 176 & 225 & 280 & 341 \\
\hline
\end{tabular}
\vspace*{1mm}
\caption{{\normalsize Values of $\mathrm{EDdegree}(\mathcal{L}_{2}^S)$ for $S=\{(1,1)\}$.}}
\label{table: 1z-r2}
\end{table}

\begin{table}[htbp]
\centering
\begin{tabular}{|c||cccccccccc|}
\hline
\backslashbox{$s$}{$n$} & 3 & 4 & 5 & 6 & 7 & 8 & 9 & 10 & 11 & 12\\
\hhline{|=||==========|}
3 & 25 & 29 &  29 &  29 &  29 &  29 &  29 &  29 &  29 &  29 \\
4 & 29 & 85 &  93 &  93 &  93 &  93 &  93 &  93 &  93 &  93 \\
5 & 29 & 93 & 181 & 193 & 193 & 193 & 193 & 193 & 193 & 193 \\
6 & 29 & 93 & 193 & 313 & 329 & 329 & 329 & 329 & 329 & 329 \\
7 & 29 & 93 & 193 & 329 & 481 & 501 & 501 & 501 & 501 & 501\\
\hline
\end{tabular}
\vspace*{1mm}
\caption{{\normalsize Values of $\mathrm{EDdegree}(\mathcal{L}_{2}^S)$ for $S=\{(1,1),(2,2)\}$.}}
\label{table: 2z-r2}
\end{table}

\begin{table}
\centering
\begin{tabular}{|c||ccccc|}
\hline
\backslashbox{$s$}{$n$} & 3 & 4 & 5 & 6 & 7 \\
\hhline{|=||=====|}
3 & 30 & 62 & 66 & 66 & 66 \\
4 & 62 & 282 & 358 & 366 & 366 \\
5 & 66 & 358 & 750 & 870 & 882 \\
6 & 66 & 366 & 870 & 1434 &  1598 \\
7 & 66 & 366 & 882 & 1598 & 2334\\
\hline
\end{tabular}
\vspace*{1mm}
\caption{{\normalsize Values of $\mathrm{EDdegree}(\mathcal{L}_{2}^S)$ for $S=\{(1,1),(2,2),(3,3)\}$.}}
\label{table: 3z-r2}
\end{table}

\begin{table}[htbp]
\centering
\begin{tabular}{|c||ccccc|}
\hline
\backslashbox{$s$}{$n$} & 4 & 5 & 6 & 7 & 8 \\
\hhline{|=||=====|}
4 & 488 & 968 & 1072 & 1080 & 1080 \\
5 & 968 & 2736 & ?& ?& ?\\
\hline
\end{tabular}
\vspace*{1mm}
\caption{{\normalsize Values of $\mathrm{EDdegree}(\mathcal{L}_{2}^S)$ for $S=\{(1,1),(2,2),(3,3),(4,4)\}$.}}
\label{table: 4z-r2}
\end{table}
\endgroup

\begin{remark}
The values of $\mathrm{EDdegree}(\mathcal{L}_{2}^S)$ in Table \ref{table: 1z-r2} are known as {\em octagonal numbers}: writing $0,1,2,\ldots$ in a hexagonal spiral around $0$, then these are numbers on the line starting from $0$ and going in the direction of $1$ \cite[\href{http://oeis.org/A000567}{A000567}]{oeis}.
Also the diagonal entries of Table \ref{table: 2z-r2} form another interesting integer sequence, see \cite[\href{https://oeis.org/A081272}{A081272}]{oeis}.
\end{remark}

Our experiments suggest a formula for $\mathrm{EDdegree}(\mathcal{L}_{2}^S)$ in the square case.

\begin{conjecture}\label{conj: r2 diagonal zero pattern}
Consider the variety $\mathcal{L}_{2}^S\subset\C^{m\times m}$, where $S$ is the zero pattern $\{(1,1),\ldots,(s,s)\}$ for some $s\ge 2$.
Then for some constant $c$
\[
\mathrm{EDdegree}(\mathcal{L}_{2}^S)=3^s\,2^{s-1}(n-s)^2 + s^{s-1}(s+1)^{\ceil*{s/2}} (n-s) + c\,.
\]
\end{conjecture}

\begin{remark}
In Tables \ref{table: 1z-r2},\ref{table: 2z-r2},\ref{table: 3z-r2},\ref{table: 4z-r2} we observe the following  symmetry property for the variety $\mathcal{L}_{r}^S\subset\C^{m\times n}$, where $S=\{(1,1),\ldots,(s,s)\}$ for some $s\ge 2$:
\[
\mathrm{EDdegree}(\mathcal{L}_r^S)(m,n,r,s)=\mathrm{EDdegree}(\mathcal{L}_r^S)(n,m,r,s)\,.
\]
This identity always holds, because the two structured best rank-$r$ approximation problems are the same after relabeling variables.
\end{remark}

We conclude by performing the same experiments showed at the beginning of the section, but restricting our study to the subspace $\mathrm{Sym}_n(\C)\subset\C^{n\times n}$ of $n\times n$ symmetric matrices.
We denote again by $\mathcal{L}_2^S$ the variety of symmetric matrices of rank at most $2$ with (symmetric) zero pattern $S\subset[n]\times[n]$.
The values of $\mathrm{EDdegree}(\mathcal{L}_{2}^S)$ with respect to the diagonal zero pattern $S=\{(1,1),\ldots,(s,s)\}$ are reported in Table \ref{table: symmetric 1,2,3,4z-r2}.

\begingroup
\begin{table}[htbp]
\centering
\begin{tabular}{|c||ccccccccc|}
\hline
\backslashbox{$s$}{$n$} & 2 & 3 & 4 & 5 & 6 & 7 & 8 & 9 & 10\\
\hhline{|=||=========|}
1 & 1 & 4 & 7  & 10  & 13  & 16  & 19  & 22  & 25\\
2 & 1 & 7 & 16 & 25  & 34  & 43  & 52  & 61  & 70\\
3 &   & 4 & 31 & 58  & 85  & 112 & 139 & 166 & 193\\
4 &   &   & 28 & 109 & 190 & 271 & 352 & 433 & 514\\
\hline
\end{tabular}
\vspace*{1mm}
\caption{{\normalsize Values of $\mathrm{EDdegree}(\mathcal{L}_{2}^S)\subset\mathrm{Sym}_n(\C)$ for $S=\{(1,1),\ldots,(s,s)\}$.}}
\label{table: symmetric 1,2,3,4z-r2}
\end{table}
\endgroup

\begin{conjecture}\label{conj: ED-1,2,3,4z-r2 symmetric}
Consider the variety $\mathcal{L}_{2}^S\subset\mathrm{Sym}_n(\C)$ with $S=\{(1,1),\ldots,(s,s)\}$ for all $s\in[4]$.
Then
\[
\mathrm{EDdegree}(\mathcal{L}_{2}^S)=
\begin{cases}
3(n-1)-2 & \mbox{if $s=1$}\\
9(n-2)-2 & \mbox{if $s=2$} \\
27(n-3)+4 & \mbox{if $s=3$} \\
81(n-4)+28 & \mbox{if $s=4$}\,.
\end{cases}
\]
\end{conjecture}

\section{Nonnegative low-rank matrix approximation}\label{sec: nonnegative}

In this section, we apply rank-two approximation with zeros to the problem of nonnegative rank-two approximation. Our goal is to find the best nonnegative rank-two approximation with a guarantee that we have found the correct solution.
There are two options for the critical points of the Euclidean distance function over $\sM_2$:
\begin{enumerate}
	\item A critical point of the Euclidean distance function over $\sM_2$ is a critical point of the Euclidean distance function over the set $\mathcal{X}_2$ of matrices of rank at most two.
	\item A critical point of the Euclidean distance function over $\sM_2$ lies on the boundary of $\sM_2$, i.e. the critical point contains one or more zero entries.
\end{enumerate}

In Example~\ref{example:3x3}, we consider $3 \times 3$ matrices and show that computing the Euclidean distance to $756$ points guarantees finding the best nonnegative rank-two approximation.
This example together with the ED degree computations in Section~\ref{sec: EDdegree} suggests that the exact nonnegative rank-two approximation problem is highly nontrivial, although in general problems on nonnegative decompositions tend to be easier for decompositions of size at most $2$ or $3$.

Recall, that for a $m \times n$ matrix $M$ there is an algorithm for nonnegative factorization with complexity $O((mn)^{O(r^2 2^r)})$, where $r$ is the nonnegative rank \cite{arora2016computing}, see also \cite{moitra2016almost}.
We can also compute a rank-$r$ approximate nonnegative factorization, under the Frobenius norm $\|M\|_{F}$, in $2^{\text{poly}(r \lg(1/\epsilon))}$ with relative error $O(\epsilon^{1/2} r^{1/4})$ \cite{arora2016computing}.
We notice that both algorithms run in polynomial time when the rank is fixed. Nevertheless, their implementation is far from straightforward and the exact constants hidden in the big-O notation could be rather big.

We end this section with simulations demonstrating that most of the time, the optimal solution is given by a critical point with a few zeros.
The reader might wonder whether it is interesting to consider the $3 \times 3$ case, when in practical applications much larger matrices are considered.
We believe that thoroughly understanding small cases is important for understanding the structure of the problem, and might provide insights for developing better numerical algorithms.

In the following example, we demonstrate that in the case of $3 \times 3$-matrices, the critical points of the rank-two approximation with zeros can be often described explicitly.

\begin{example}
Let $U=(u_{ij})\in\R^{3\times 3}$. We consider the best rank-two approximation problem with zeros in $S=\{(1,1),(1,2),(2,2)\}$.
Then there are three critical points, each of which has in addition to the entries in $S$ one further entry equal to zero and all other entries are equal to the corresponding entries of the matrix $U$.
Specifically, the critical points are
\[
\begin{pmatrix}
0&0&0\\
u_{21} & 0 & u_{23}\\
u_{31} & u_{32} & u_{33}
\end{pmatrix}\,,
\quad
\begin{pmatrix}
0&0&u_{13}\\
0 & 0 & u_{23}\\
u_{31} & u_{32} & u_{33}
\end{pmatrix}\,,
\quad
\begin{pmatrix}
0&0&u_{13}\\
u_{21} & 0 & u_{23}\\
u_{31} & 0 & u_{33}
\end{pmatrix}\,.
\]
The additional entries that are set to zero can be read from the determinant of the matrix
\[
\begin{pmatrix} 0 & 0 & x_{13}\\ x_{21} & 0 & x_{23} \\ x_{31} & x_{32} & x_{33} \end{pmatrix}\,,
\]
that is $x_{13} x_{21} x_{32}$.
Its factors correspond precisely to the additional entries that are set to zero to obtain the three critical points.

In other words, we consider the Euclidean distance minimization problem to the variety defined by $x_{13} x_{21} x_{32}$.
Since this variety is reducible, we can consider the Euclidean distance minimization problem to each of its three irreducible components.
Each of the irreducible components has ED degree 1.

This example generalizes to $n \times n$ matrices whose non-zero entries form a lower triangular submatrix.
The determinant of such a matrix is the product of the diagonal elements, and hence the critical points of the best rank $(n-1)$-approximation problem are obtained by adding a zero to the diagonal.

A more interesting example is given when $S=\{(1,1),(1,2)\}$.
In this case, the determinant is $x_{13}(-x_{22} x_{31} + x_{21} x_{32})$.
One of the critical points has the entry $x_{13}$ equal to zero and other entries equal to the corresponding entries of $U$.
The two other critical points agree with $U$ in the third column and the $2 \times 2$ submatrix defined by the rows $2,3$ and columns $1,2$ is equal to one of the critical points of the rank-one approximation for the corresponding $2 \times 2$ submatrix of $U$.

This example generalizes to a zero pattern that contains all but one entry in a row or in a column.
Then the determinant factors as a variable times a $(n-1) \times (n-1)$ subdeterminant.
One critical point is obtained by adding the missing zero and the rest of the critical points are obtained by rank $(n-2)$-approximations for the $(n-1) \times (n-1)$ submatrix whose determinant is a factor in the above product.
\end{example}

In the following example, we discuss how to find the best nonnegative rank-two estimate of a $3 \times 3$-nonnegative matrix with guarantee.

\begin{example}\label{example:3x3}
Consider the group whose elements are simultaneous permutations of rows and columns of a $3 \times 3$ matrix, and permutations of rows with columns.
This group acts on the set zero patterns of a $3 \times 3$ matrix.
There are $26$ orbits of this group action, $13$ of which are listed in Table~\ref{table: 3x3-zero-patterns}.
The columns of Table~\ref{table: 3x3-zero-patterns} list an orbit representative, the orbit size, the ED degree and the description of critical points if available.

The $13$ orbit representatives listed in Table~\ref{table: 3x3-zero-patterns} have the property that there is no zero pattern $S$ with less zeros such that the zero pattern of a critical point for $S$ is contained in the orbit representative.
For example, no zero pattern that contains a row or a column is listed in Table~\ref{table: 3x3-zero-patterns}, because a critical point on the line three of the table has a row of zeros.
If a critical point agrees with $U$ at all its non-zero entries, then adding more zeros causes the Euclidean distance to the data matrix to increase, so such critical points can be discarded.

Moreover, we can also discard the seven orbits of zero patterns marked with star in Table~\ref{table: 3x3-zero-patterns}, because their critical points either appear earlier in the table or the zero patterns of critical points contain the zero pattern of a critical point that appears earlier in the table.

In summary, there are five different kinds of critical points to be considered:
\begin{enumerate}
\item Sum of any 2 components of the SVD of $U$.
In total: $1 \cdot 3=3$.
\item Critical points of diagonal zero patterns.
In total:  $9 \cdot 8 + 18 \cdot 25 + 6 \cdot 30=702$.
\item Critical points that set one row or column of $U$ to zero.
In total:  $6 \cdot 1=6$.
\item Critical points where a $2 \times 2$-submatrix is given by a rank-one critical point of the corresponding submatrix of $U$.
These critical points also have two zeros and one row or column equal to the corresponding row or column of $U$. In total: $18 \cdot 2 = 36$.
\item Critical points where zeros form a $2 \times 2$-submatrix.
In total: $9 \cdot 1 =9$.
\end{enumerate}
In total, the number of critical points is $3+702+6+36+9=756$.
Thus, given a nonnegative $3 \times 3$-matrix $U$, if we construct the $756$ critical points described above and choose among the nonnegative critical points the one that is closest to $U$, then it is guaranteed to be the best nonnegative rank-two approximation of the matrix $U$.

\begingroup
\setlength{\tabcolsep}{5pt}
\renewcommand{\arraystretch}{1.7}
\begin{table}[htbp]
\centering
\begin{tabular}{|c||c|c|c|c|}
\hline
& $S$ & \texttt{\#}orbit & $\mathrm{EDdegree}(\mathcal{L}_r^S)$ & critical points \\
\hhline{|=||====|}
1 & {\tiny$\begin{pmatrix} \cdot & \cdot & \cdot \\ \cdot & \cdot & \cdot \\ \cdot & \cdot & \cdot \end{pmatrix}$} & 1 & 3 &  sum of any 2 components of SVD\\
\hline
2 & {\tiny $\begin{pmatrix} 0 & \cdot & \cdot \\ \cdot & \cdot & \cdot \\ \cdot & \cdot & \cdot \end{pmatrix}$} & 9 & 8 & no interpretation\\
\hline
3 & {\tiny $\begin{pmatrix} 0 & 0 & \cdot \\ \cdot & \cdot & \cdot \\ \cdot & \cdot & \cdot \end{pmatrix}$} & 18 & 3 & {\tiny $\begin{pmatrix} 0 & 0 & 0 \\ \cdot & \cdot & \cdot \\ \cdot & \cdot & \cdot \end{pmatrix}$} or $\rank(X_{\{2,3\},\{1,2\}})=1$\\
\hline
4 & {\tiny $\begin{pmatrix} 0 & \cdot & \cdot \\ \cdot & 0 & \cdot \\ \cdot & \cdot & \cdot \end{pmatrix}$} & 18 & 25 & no interpretation\\
\hline
5* & {\tiny $\begin{pmatrix} 0 & 0 & \cdot \\ 0 & \cdot & \cdot \\ \cdot & \cdot & \cdot \end{pmatrix}$} & 36 & 3 & {\tiny $\begin{pmatrix} 0 & 0 & \cdot \\ 0 & 0 & \cdot \\ \cdot & \cdot & \cdot \end{pmatrix}$} or {\tiny $\begin{pmatrix} 0 & 0 & 0 \\ 0 & \cdot & \cdot \\ \cdot & \cdot & \cdot \end{pmatrix}$} or {\tiny $\begin{pmatrix} 0 & 0 & \cdot \\ 0 & \cdot & \cdot \\ 0 & \cdot & \cdot \end{pmatrix}$}\\
\hline
6* & {\tiny $\begin{pmatrix} 0 & 0 & \cdot \\ \cdot & \cdot & 0 \\ \cdot & \cdot & \cdot \end{pmatrix}$} & 36 & 3 & {\tiny $\begin{pmatrix} 0 & 0 & 0 \\ \cdot & \cdot & 0 \\ \cdot & \cdot & \cdot \end{pmatrix}$} or $\rank(X_{\{2,3\},\{1,2\}})=1$\\
\hline
7 & {\tiny $\begin{pmatrix} 0 & \cdot & \cdot \\ \cdot & 0 & \cdot \\ \cdot & \cdot & 0 \end{pmatrix}$} & 6 & 30 & no interpretation\\
\hline
8 & {\tiny $\begin{pmatrix} 0 & 0 & \cdot \\ 0 & 0 & \cdot \\ \cdot & \cdot & \cdot \end{pmatrix}$} & 9 & 1 & projection onto $\mathcal{L}^S$\\
\hline
9* & {\tiny $\begin{pmatrix} 0 & 0 & \cdot \\ 0 & \cdot & 0 \\ \cdot & \cdot & \cdot \end{pmatrix}$} & 36 & 3 & {\tiny $\begin{pmatrix} 0 & 0 & 0 \\ 0 & \cdot & 0 \\ \cdot & \cdot & \cdot \end{pmatrix}$} or {\tiny $\begin{pmatrix} 0 & 0 & \cdot \\ 0 & 0 & 0 \\ \cdot & \cdot & \cdot \end{pmatrix}$} or {\tiny $\begin{pmatrix} 0 & 0 & \cdot \\ 0 & \cdot & 0 \\ 0& \cdot & \cdot \end{pmatrix}$}\\
\hline
10* & {\tiny $\begin{pmatrix} 0 & 0 & \cdot \\ 0 & \cdot & \cdot \\ \cdot & \cdot & 0 \end{pmatrix}$} & 36 & 3 & {\tiny $\begin{pmatrix} 0 & 0 & 0 \\ 0 & \cdot & \cdot \\ \cdot & \cdot & 0 \end{pmatrix}$} or {\tiny $\begin{pmatrix} 0 & 0 & \cdot \\ 0 & 0 & \cdot \\ \cdot & \cdot & 0 \end{pmatrix}$} or {\tiny $\begin{pmatrix} 0 & 0 & \cdot \\ 0 & \cdot & \cdot \\ 0 & \cdot & 0
\end{pmatrix}$}\\
\hline
11* & {\tiny $\begin{pmatrix} 0 & 0 & \cdot \\ \cdot & \cdot & 0 \\ \cdot & \cdot & 0 \end{pmatrix}$} & 9 & 3 & {\tiny $\begin{pmatrix} 0 & 0 & 0 \\ \cdot & \cdot & 0 \\ \cdot & \cdot & 0 \end{pmatrix}$} or $\rank(X_{\{2,3\},\{1,2\}})=1$\\
\hline
12* & {\tiny $\begin{pmatrix} 0 & 0 & \cdot \\ 0 & \cdot & 0 \\ \cdot & 0 & \cdot \end{pmatrix}$} & 36 & 3 & {\tiny $\begin{pmatrix} 0 & 0 & 0 \\ 0 & \cdot & 0 \\ \cdot & 0 & \cdot \end{pmatrix}$} or
{\tiny $\begin{pmatrix} 0 & 0 & \cdot \\ 0 & 0 & 0 \\ \cdot & 0 & \cdot \end{pmatrix}$} or {\tiny $\begin{pmatrix} 0 & 0 & \cdot \\ 0 & \cdot & 0 \\ 0 & 0 & \cdot \end{pmatrix}$}\\
\hline
13* & {\tiny $\begin{pmatrix} 0 & 0 & \cdot \\ 0 & \cdot & 0 \\ \cdot  & 0 & 0 \end{pmatrix}$} & 6 & 3 & {\tiny $\begin{pmatrix} 0 & 0 & 0 \\ 0 & \cdot & 0 \\ \cdot & 0 & 0 \end{pmatrix}$} or {\tiny $\begin{pmatrix} 0 & 0 & \cdot \\ 0 & 0 & 0 \\ \cdot & 0 & 0 \end{pmatrix}$} or {\tiny $\begin{pmatrix} 0 & 0 & \cdot \\ 0 & \cdot & 0 \\ 0 & 0 & 0 \end{pmatrix}$}\\
\hline
\end{tabular}
\vspace*{1mm}
\caption{The $13$ orbit representatives of zero patterns of $3 \times 3$ matrices that have the property that there is no zero pattern $S$ with less zeros such that the zero pattern of a critical point for $S$ is contained in the orbit representative.}
\label{table: 3x3-zero-patterns}
\end{table}
\endgroup
\end{example}

Example~\ref{example:3x3} suggests that finding the best nonnegative rank-two approximation of a generic matrix with guarantee might be hopeless, because we expect the number of critical points to increase at least exponentially in the matrix size by the conjectures in Section~\ref{sec: EDdegree}. In practice, the best rank-two approximation often has a few zeros.

Using Macaulay2 we sampled uniformly randomly $10^5$ matrices from the set of $3 \times 3$ matrices with real nonnegative entries and the sum of entries being equal to $1000$ and for each data sample we computed numerically a best nonnegative rank-two approximation.
In $88561$ cases, the best approximation has no zeros; in $10550$ cases, the best approximation has one zero; in $889$ cases, the best approximation has two zeros in different rows and columns.
Based on this experiment, we observed two interesting phenomena:
\begin{itemize}
    \item[$(a)$] We never encountered a best nonnegative rank-two approximation with three zeros or with two aligned zeros.
    \item[$(b)$] If the best nonnegative rank-two approximation has zero pattern $S$, then the best rank-two approximation given by SVD has negative entries in $S$.
\end{itemize}
These facts lead to the following open questions.
\begin{question}
\begin{enumerate}
    \item Are the experimental observations $(a)$ and $(b)$ true for any nonnegative matrix $U\in\R^{3\times 3}$?
    \item Given a nonnegative matrix $U\in\R_{\mathsmaller{\ge 0}}^{m\times n}$ whose best nonnegative rank-$2$ approximation has zero pattern $S$, does the best rank-$2$ approximation given by SVD have negative entries in $S$?
    \item For which zero patterns $S\subset[m]\times[n]$ and target ranks $r\in[m-1]$ there exists a full rank nonnegative matrix $U\in\R_{\mathsmaller{\ge 0}}^{m\times n}$ whose best nonnegative rank-$r$ approximation has zero pattern $S$?
\end{enumerate}
\end{question}

\section*{Acknowledgements}
\noindent We thank Giorgio Ottaviani, Grégoire Sergeant-Perthuis, Pierre-Jean Spaenlehauer, and Bernd Sturmfels for helpful discussions and suggestions. We thank two anonymous reviewers for insightful comments which improved the original manuscript.
Kaie Kubjas and Luca Sodomaco are partially supported by the Academy of Finland
Grant No.~323416.
Elias Tsigaridas is partially supported by ANR JCJC GALOP
(ANR-17-CE40-0009), the PGMO grant ALMA, and the PHC GRAPE.

\bibliographystyle{alpha}
\bibliography{low_rank_approximation_with_zeros}

\newcommand{\etalchar}[1]{$^{#1}$}
\begin{thebibliography}{AGKM16}

\bibitem[AB09]{ArBa-cc-2009}
Sanjeev Arora and Boaz Barak.
\newblock {\em Computational complexity: a modern approach}.
\newblock Cambridge University Press, 2009.

\bibitem[AGKM16]{arora2016computing}
Sanjeev Arora, Rong Ge, Ravi Kannan, and Ankur Moitra.
\newblock Computing a nonnegative matrix factorization---provably.
\newblock {\em SIAM J. Comput.}, 45(4):1582--1611, 2016.

\bibitem[BT18]{HomotopyContinuation.jl}
Paul Breiding and Sascha Timme.
\newblock {H}omotopy{C}ontinuation.jl: {A} {P}ackage for {H}omotopy
  {C}ontinuation in {J}ulia.
\newblock In {\em International Congress on Mathematical Software}, pages
  458--465. Springer, 2018.

\bibitem[CFP03]{chu2003structured}
Moody~T. Chu, Robert~E. Funderlic, and Robert~J. Plemmons.
\newblock Structured low rank approximation.
\newblock {\em Linear Algebra Appl.}, 366:157--172, 2003.

\bibitem[CH87]{CH}
Dipa Choudhury and Roger~A. Horn.
\newblock An analog of the singular value decomposition for complex orthogonal
  equivalence.
\newblock {\em Linear Multilinear Algebra}, 21(2):149--162, 1987.

\bibitem[Cif21]{cifuentes2021convex}
Diego Cifuentes.
\newblock A convex relaxation to compute the nearest structured rank deficient
  matrix.
\newblock {\em SIAM J. Matrix Anal. Appl.}, 42(2):708--729, 2021.

\bibitem[CLO92]{cox1992ideals}
David Cox, John Little, and Donal O'Shea.
\newblock {\em Ideals, varieties, and algorithms}.
\newblock Undergraduate Texts in Mathematics. Springer-Verlag, New York, 1992.
\newblock An introduction to computational algebraic geometry and commutative
  algebra.

\bibitem[CSS13]{CSS}
Sergio Caracciolo, Alan~D. Sokal, and Andrea Sportiello.
\newblock Algebraic/combinatorial proofs of {C}ayley-type identities for
  derivatives of determinants and {P}faffians.
\newblock {\em Adv. in Appl. Math.}, 50(4):474--594, 2013.

\bibitem[CW19]{ConWel-lovasz-19}
Aldo Conca and Volkmar Welker.
\newblock {L}ov{\'a}sz--{S}aks--{S}chrijver ideals and coordinate sections of
  determinantal varieties.
\newblock {\em Algebra \& Number Theory}, 13(2):455--484, 2019.

\bibitem[DHO{\etalchar{+}}16]{draisma2016euclidean}
Jan Draisma, Emil Horobe\c{t}, Giorgio Ottaviani, Bernd Sturmfels, and Rekha~R.
  Thomas.
\newblock The {E}uclidean distance degree of an algebraic variety.
\newblock {\em Found. Comput. Math.}, 16(1):99--149, 2016.

\bibitem[DLOT17]{DLOT}
Dmitriy Drusvyatskiy, Hon-Leung Lee, Giorgio Ottaviani, and Rekha~R. Thomas.
\newblock The {E}uclidean distance degree of orthogonally invariant matrix
  varieties.
\newblock {\em Israel J. Math.}, 221(1):291--316, 2017.

\bibitem[DOT18]{DOT}
Jan Draisma, Giorgio Ottaviani, and Alicia Tocino.
\newblock Best rank-{$k$} approximations for tensors: generalizing
  {E}ckart-{Y}oung.
\newblock {\em Res. Math. Sci.}, 5(2):27, 2018.

\bibitem[GG11]{GilGli-low-2011}
Nicolas Gillis and Fran{\c{c}}ois Glineur.
\newblock Low-rank matrix approximation with weights or missing data is
  np-hard.
\newblock {\em SIAM Journal on Matrix Analysis and Applications},
  32(4):1149--1165, 2011.

\bibitem[GHS87]{golub1987generalization}
G.~H. Golub, Alan Hoffman, and G.~W. Stewart.
\newblock A generalization of the {E}ckart-{Y}oung-{M}irsky matrix
  approximation theorem.
\newblock {\em Linear Algebra Appl.}, 88/89:317--327, 1987.

\bibitem[Gil20]{gillis2020nonnegative}
Nicolas Gillis.
\newblock {\em Nonnegative Matrix Factorization}.
\newblock SIAM, 2020.

\bibitem[GS]{M2}
Daniel~R. Grayson and Michael~E. Stillman.
\newblock Macaulay2, a software system for research in algebraic geometry.
\newblock Available at
  \href{http://www.math.uiuc.edu/Macaulay2/}{www.math.uiuc.edu/Macaulay2/}.

\bibitem[HR20]{horobet2020data}
Emil Horobet and Jose~Israel Rodriguez.
\newblock Data loci in algebraic optimization.
\newblock {\em \arxiv{2007.04688}}, 2020.

\bibitem[Lal96]{Lalonde}
Pierre Lalonde.
\newblock A non-commutative version of {J}acobi's equality on the cofactors of
  a matrix.
\newblock {\em Discrete Math.}, 158(1-3):161--172, 1996.

\bibitem[Map19]{maple}
Maple.
\newblock {\em Maplesoft a division of Waterloo Maple Inc: user manual (ver.
  2019)}.
\newblock Waterloo, Ontario, 2019.

\bibitem[Mar08]{markovsky2008structured}
Ivan Markovsky.
\newblock Structured low-rank approximation and its applications.
\newblock {\em Automatica}, 44(4):891--909, 2008.

\bibitem[Mar19]{markovsky2012low}
Ivan Markovsky.
\newblock {\em Low-rank approximation}.
\newblock Communications and Control Engineering Series. Springer, Cham, 2019.

\bibitem[Moi16]{moitra2016almost}
Ankur Moitra.
\newblock An almost optimal algorithm for computing nonnegative rank.
\newblock {\em SIAM J. Comput.}, 45(1):156--173, 2016.

\bibitem[OP15]{MR3456581}
Giorgio Ottaviani and Raffaella Paoletti.
\newblock A geometric perspective on the singular value decomposition.
\newblock {\em Rend. Istit. Mat. Univ. Trieste}, 47:107--125, 2015.

\bibitem[OSS14]{OSS-esslra-14}
Giorgio Ottaviani, Pierre-Jean Spaenlehauer, and Bernd Sturmfels.
\newblock Exact solutions in structured low-rank approximation.
\newblock {\em SIAM J. Matrix Anal. Appl.}, 35(4):1521--1542, 2014.

\bibitem[PB83]{provan1983complexity}
J.~Scott Provan and Michael~O. Ball.
\newblock The complexity of counting cuts and of computing the probability that
  a graph is connected.
\newblock {\em SIAM J. Comput.}, 12(4):777--788, 1983.

\bibitem[Pie15]{piene2015polar}
Ragni Piene.
\newblock Polar varieties revisited.
\newblock In {\em Computer algebra and polynomials}, volume 8942 of {\em
  Lecture Notes in Comput. Sci.}, pages 139--150. Springer, Cham, 2015.

\bibitem[PS17]{ParSel-improved-17}
Ankit Parekh and Ivan~W Selesnick.
\newblock Improved sparse low-rank matrix estimation.
\newblock {\em Signal Processing}, 139:62--69, 2017.

\bibitem[Slo]{oeis}
Neil J.~A. Sloane.
\newblock The on-line encyclopedia of integer sequences.
\newblock Available at \href{http://oeis.org}{oeis.org}.

\bibitem[Sod20]{sodphd}
Luca Sodomaco.
\newblock {\em The Distance Function from the Variety of partially symmetric
  rank-one Tensors}.
\newblock PhD thesis, Universit\`a degli Studi di Firenze, 2020.
\newblock Available at
  \href{https://sites.google.com/view/luca-sodomaco/home/publications?authuser=0}{https://sites.google.com/view/luca-sodomaco/home/publications}.

\bibitem[SRV12]{SalRicVay-icml-012}
Pierre-Andr{\'e} Savalle, Emile Richard, and Nicolas Vayatis.
\newblock Estimation of simultaneously sparse and low rank matrices.
\newblock In {\em ICML}, 2012.

\bibitem[SS16]{schost2016quadratically}
{\'E}ric Schost and Pierre-Jean Spaenlehauer.
\newblock A quadratically convergent algorithm for structured low-rank
  approximation.
\newblock {\em Found. Comput. Math.}, 16(2):457--492, 2016.

\end{thebibliography}

\end{document}